\documentclass[10pt]{article}
\usepackage[nointlimits]{amsmath}
\usepackage{amsfonts}\usepackage{amssymb}\usepackage{amsmath}\usepackage{amsthm}
\usepackage{latexsym}
\usepackage{graphicx}
\usepackage{layout}
\usepackage{color}

\newtheorem{theorem}{Theorem}[section]
\newtheorem{definition}[theorem]{Definition}

\newtheorem{proposition}[theorem]{Proposition}
\newtheorem{lemma}[theorem]{Lemma}
\newtheorem{corollary}[theorem]{Corollary}
\newtheorem{remark}[theorem]{Remark}

\newcommand{\sezione}[1]{\section{#1}\setcounter{equation}{0}}

\newcommand{\nor}{\Arrowvert}

\def\R{{\rm I\mskip -3.5mu R}}
\def\N{{\rm I\mskip -3.5mu N}}

\def\cS{{\mathcal{S}}}
\def\cH{{\mathcal{H}}}
\def\cSn{{\mathcal{S}(n)}}
\def\e{{\varepsilon}}
\def\di12{\mathcal{D}^{1,2}(\R^n)}

\def\na{\nabla}
\def\d{\delta}
\def\g{\gamma}
\def\D{\Delta}
\def\l{{\lambda}}
\def\L{{\Lambda}}
\def\a{{\alpha}}
\def\b{{\beta}}

\def\o{\omega}

\def\0l{_{0,\l}}
\def\1l{_{1,\l}}
\def\2l{_{2,\l}}
\def\3l{_{3,\l}}
\def\4l{_{4,\l}}

\def\de{\partial}

\def\Om{\Omega}

\def\c{{\gamma}}

\begin{document}
\title{
Nonradial solutions  for the H\'enon equation in $\R^N$
\thanks{The first two authors are supported by PRIN-2009-WRJ3W7 grant,
S. Neves has been partially supported by FAPESP.}}
\author{Francesca Gladiali  \thanks{Dipartimento Polcoming-Matematica e Fisica, Universit\`a di
Sassari,via Piandanna 4 -07100 Sassari, e-mail {\sf
fgladiali@uniss.it}.} \and Massimo Grossi\thanks{Dipartimento di Matematica, Universit\`a di Roma
La Sapienza, P.le A. Moro 2 - 00185 Roma, e-mail {\sf
massimo.grossi@uniroma1.it}.} \and S\'ergio L. N. Neves \thanks{Departamento de Matem\'atica, Universidade Estadual de Campinas - IMECC,
Rua S\'ergio Buarque de Holanda 651, Campinas-SP 13083-859, Brazil,
e-mail {\sf sergio184@gmail.com}}}
\maketitle
 \begin{abstract}
In this paper we consider the problem
$$
\left\{\begin{array}{ll}
-\Delta u=(N+\a)(N-2)|x|^{\a}u^\frac{N+2+2\a}{N-2}  & \hbox{ in }\R^N\\
u>0& \hbox{ in }\R^N\\
u\in D^{1,2}\left(\R^N\right)
\end{array}\right.
$$
where $N\ge3$. From the characterization of the solutions of the linearized operator, we deduce the existence of nonradial solutions which bifurcate from the radial one when $\alpha$ is an even integer.

 \end{abstract}

\tableofcontents

\sezione{Introduction and statement of the main results}
We consider the problem
\begin{equation}\label{1}
\left\{\begin{array}{ll}
-\Delta u=C(\a)|x|^{\a}u^{p_\a}  & \hbox{ in }\R^N\\
u>0& \hbox{ in }\R^N\\
u\in D^{1,2}\left(\R^N\right)
\end{array}\right.
\end{equation}
where $N\geq 3$, $\a> 0$, $p_\a=\frac{N+2+2\a}{N-2}$,
$C(\a)=(N+\a)(N-2)$,
$D^{1,2}\left(\R^N\right)=\{u\in L^{2^*}(\R^N)\hbox{ such that } |\na
u| \in L^2(\R^N)\}$ and $2^*=\frac{2N}{N-2}$.\\
This problem, for $\alpha>0$, generalizes the well-known equation which involves the critical Sobolev exponent
\begin{equation}\label{i2}
\left\{\begin{array}{ll}
-\Delta u=N(N-2)u^\frac{N+2}{N-2}  & \hbox{ in }\R^N\\
u>0& \hbox{ in }\R^N.
\end{array}\right.
\end{equation}
Smooth solutions to \eqref{i2} have been completely classified in  \cite{CGS89}, where the authors proved that they are given by
\begin{equation}\label{i3}
U_{\l,y}(x)=\frac {\l^{\frac{N-2}2}}{\left(1+\l^{2}|x-y|^2\right)^\frac{N-2}2}
\end{equation}
with $\l>0$ and $y\in\R^N$
and they are extremal functions for the well-known Sobolev inequality,
\begin{equation}\label{i3a}
\int_{\R^N}|\nabla u|^2\ge S\left(\int_{\R^N}|u|^\frac{2N}{N-2}\right)^\frac{N-2}N.
\end{equation}
The presence of the term $|x|^{\a}$ in equation \eqref{1} drastically changes the problem. For this kind of nonlinearities it is not possible to apply the moving plane method anymore (to get the radial symmetry around some point), and indeed nonradial solutions appear, as we will see in Theorem \ref{i17}. This phenomenon has brought attention to the H\'enon problem, i.e. \eqref{1} or
\begin{equation}\label{i5}
\left\{\begin{array}{ll}
-\Delta u=C(\a)|x|^{\a}u^{p}  & \hbox{ in }B_1\\
u>0& \hbox{ in }B_1\\
u=0& \hbox{ on }\partial B_1
\end{array}\right.
\end{equation}
where $B_1$ is the unit ball  of $\R^N$, $N\ge3$, and $p>1$.
Problem  \eqref{i5} was widely studied, mainly in the subcritical range $1<p<\frac{N+2}{N-2}$ (see for example  \cite{SSW02},  \cite{PS07}
 and the references therein) where the existence of nonradial solutions was observed. The only result in the full range $(1,\frac{N+2+2\alpha}{N-2})$ is the following one (see W. M. Ni,  \cite{N82}) ,
 \begin{theorem}\label{i4}
Let $B_1$ be the unit ball of $\R^N$, $N\ge3$. Then, for any $1<p<\frac{N+2+2\alpha}{N-2}$ there exists a radial solution to the  problem \eqref{i5}.
Moreover, if $p\ge\frac{N+2+2\alpha}{N-2}$ there exists no solution to \eqref{i5}.
\end{theorem}
Coming back to  \eqref{1}, we quote the following result by E. Lieb (\cite{L83}), which, in the radial case, extends the inequality  \eqref{i3a}.
\begin{theorem}\label{i6}
 Let $u\in D^{1,2}\left(\R^N\right)$ be a radial function. Then we have that,
\begin{equation}\label{i7}
\int_{\R^N}|\nabla u|^2\ge S(\alpha)\left(\int_{\R^N}|x|^\alpha|u|^\frac{2N+2\alpha}{N-2}\right)^\frac{N-2}N,
\end{equation}
for some positive constant $S(\a)$. Moreover the extremal functions which achieve $S(\alpha)$ are solutions to \eqref{1} and are given by
\begin{equation}\label{i8}
U_{\l,\a}(x)=\frac {\l^{\frac{N-2}2}}{\left(1+\l^{2+\a}|x|^{2+\a}\right)^{\frac{N-2}{2+\a}}}
\end{equation}
with $\l>0$.
\end{theorem}
In (\cite{GS81}) it was proved that the functions in \eqref{i8} are the $unique$ radial solutions to \eqref{1}.
We will call $U_{\a}$ the unique radial solution of \eqref{1}, related to the exponent $\a$, such that $U_{\a}(0)=1$, i.e.
\begin{equation}
\label{i9}
U_{\a}(x)=\frac 1{\left(1+|x|^{2+\a}\right)^\frac{N-2}{2+\a}}.
\end{equation}
In this paper we are interested in the existence of nonradial solutions for problem  \eqref{1}. This problem is quite difficult because, in this case, there is no embedding of the space $D^{1,2}\left(\R^N\right)$ in $L^\frac{2N+2\a}{N-2}\left(\R^N\right)$. So the standard variational methods can not be applied. To overcome this problem we will use the {\em bifurcation theory}.\\
Our first result concerns the study of the linearized problem related to \eqref{1} at the function $U_{\a}$. This leads to study the problem,
\begin{equation}\label{1.4}
\left\{\begin{array}{ll}
-\Delta v=C(\a)p_\a|x|^{\a}U_{\a}^{p_\a-1}v  & \hbox{ in }\R^N\\
v\in D^{1,2}\left(\R^N\right).
\end{array}\right.
\end{equation}
Next theorem characterizes all the solutions to \eqref{1.4}.
\begin{theorem} \label{linearized}
Let $\a \geq 0$. If $\a > 0$ is not an even integer, then the space of solutions of \eqref{1.4} has dimension $1$ and is spanned by
\begin{equation} \label{i12}
Z(x)=\frac{1-|x|^{2+\a}}{(1+|x|^{2+\a})^{\frac{N+\a}{2+\a}}}.
\end{equation}
If $\a = 2(k-1)$ for some $k \in \N$ then the space of solutions of  \eqref{1.4} has dimension $1 + \frac{(N+2k-2)(N+k-3)!}{(N-2)!k!}$
and is spanned by the functions
\begin{equation} \label{i13}
Z(x)=\frac{1-|x|^{2+\a}}{(1+|x|^{2+\a})^{\frac{N+\a}{2+\a}}} \,\, , \,\, Z_k(x)=\frac{Y_k(x)}{(1+|x|^{2+\a})^{\frac{N+\a}{2+\a}}}
\end{equation}
where $Y_k$ form a basis of $\mathbb{Y}_k(\R^N)$, the space of all homogeneous harmonic polynomials of degree $k$ in $\R^N$.
\end{theorem}

We note that in the case $\alpha=0$ we get $k = 1$ and one gets
back the known result for the equation involving the critical
Sobolev exponent. We observe that for all $\alpha>0$ the problem
\eqref{1} is invariant for dilations but not for translations.
Theorem \ref{linearized} highlights the new phenomenon that if $\alpha$
is an even integer then there exist new solutions to \eqref{1.4} that
``replace" the ones due to the translations invariance.
It would be very interesting to understand if these new solutions are given by some geometrical invariants of the problem or not.\\
The key step of the proof is the change of variables $r\mapsto r^\frac2{\alpha+2}$. In this way the problem  \eqref{1.4} leads back, in a suitable sense, to the well-known case $\alpha = 0$, where there is a complete characterization of the solutions.\\
We emphasize that the transformation $r\mapsto r^\frac2{\alpha+2}$, which was used in \cite{CG10} in a different context, allows to prove in an easy way some known results.\\
The first example is a new (and in our opinion very simple) proof of the inequality \eqref{i7}, which also provides the uniqueness result due to Gidas and Spruck in
\cite{GS81}. The second one is a new proof of Theorem \ref{i4} jointly with the uniqueness of the radial solution (this last result was proved in \cite{NN85}). Both proofs are given in the Appendix.\\
A first consequence of Theorem \ref{linearized} is the computation of the Morse index of the solution $U_{\a}$.
\begin{corollary}\label{cor-1}
Let $U_{\a}$ be the solution of \eqref{1}, then its Morse index $m(\a)$ is equal to
$$m(\a) = \sum_{0 \leq  k<\frac{\a+2}{2}\atop_{k\  integer}} \frac{(N+2k-2)(N+k-3)!}{(N-2)!\,k!}.$$
In particular, we have that the Morse index of $U_{\a}$ changes as $\a$ crosses the even integers and also that $m(\a) \to+\infty$ as $\a \to+\infty.$
\end{corollary}
Now let us consider the most important consequence of Theorem \ref{linearized}: the existence of {\em nonradial
solutions} to \eqref{1}. \\
\noindent Set $X=D^{1,2}(\R^N)\cap L^{\infty}_{\b}(\R^N)$   where $L^{\infty}_{\b}(\R^N)$ is a suitable $L^{\infty}-$weighted space (see \eqref{norm-beta}, \eqref{norm-x} for the precise definition).
First let us give the following definition,
\begin{definition} \label{i16}
Let $U_{\a}$ be the radial solution of \eqref{1} defined in \eqref{i9}. We say that a nonradial bifurcation occurs at $(\overline\a,U_{\overline{\a}})$ if in every neighborhood of $(\overline\a,U_{\overline{\a}}) $ in $(0,+\infty)\times X$ there exists a point  $(\a,v_\a)$ with $v_\a$ nonradial solution of  \eqref{1}.
\end{definition}
\noindent Let $O(h)$ be the orthogonal group in $\R^h$.
Our main result is the following.

\begin{theorem} \label{i17}
Let $\a=2(k-1)$ with $k\in\mathbb{N}$, $k\geq 2$. Then, (see Figure 1)\\
i) If $k$ is odd there exists at least a continuum of nonradial solutions to \eqref{1}, invariant with respect to $O(N-1)$, bifurcating from the pair $(\a,U_\a)$.\\
ii) If $k$ is even there exist at least $\big[\frac N2\big]$ continua of nonradial solutions to
\eqref{1}, invariant with respect to $O(N-1),O(N-2)\times O(2),..\dots,$
 bifurcating from the pair $(\a,U_\a)$.\\
Moreover all these solutions are fast decaying, in the sense that
$$\limsup _{|x|\to +\infty} |x|^{N-2}v(x)<+\infty.$$
\end{theorem}
\begin{figure}[h!]
\centering
 \vskip-0.5truecm
\includegraphics[scale=0.4]{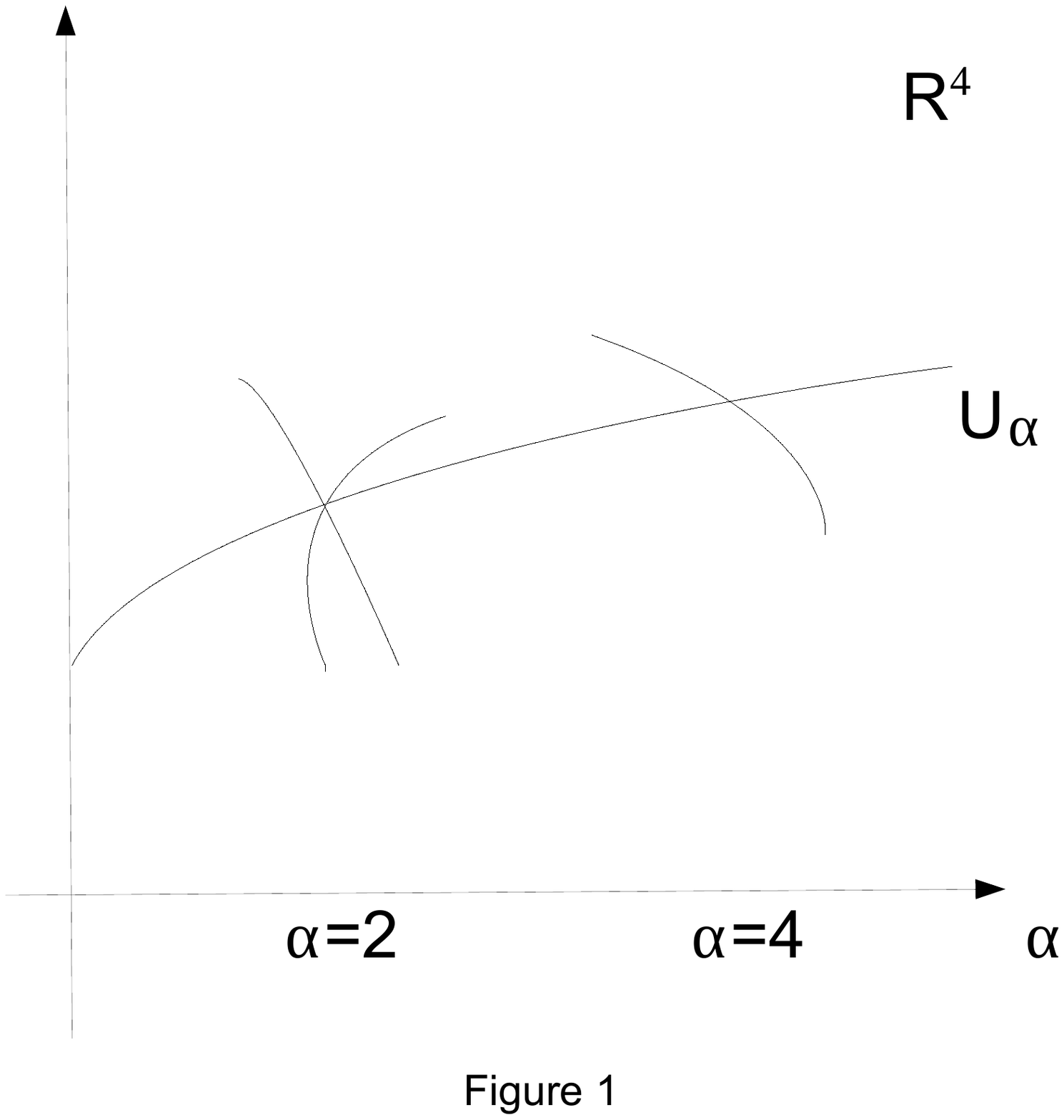}
\label{fig1}
\end{figure}
\vskip-4truecm
The previous theorem states that the structure of solutions to \eqref{1} is much more complex than the case $\a = 0$. In particular, it highlights the special role of the {\em even} numbers $\a$.
The proof of Theorem \ref{i17} requires a lot of work. In fact, even if Theorem
\ref{linearized} suggests the existence of nonradial bifurcation points, it is not possible to a
pply directly  the classical bifurcation theory, because the solution $U_\a$ is not isolated. This fact also makes very complicated to calculate the degree of the operator naturally associated with the problem, which is known as a crucial tool in the bifurcation theory.\\
In order to overcome these difficulties, we introduce a suitable approximated problem
on balls of radius $\frac1\epsilon$ and there we apply the classical bifurcation theory.
In this way we deduce the existence of nonradial solutions $v_\epsilon$ which bifurcate from some
radial functions close to $U_\a $. The final part of the proof will be to show that these
solutions converge to nonradial solutions of problem \eqref{1} as $\epsilon\rightarrow0$.
This last part requires several delicate estimates. In particular, we emphasize that a careful
use of the Pohozaev identity allow us to show that the approximated solutions  $v_\e$ stay
 away from the branches of radial solutions of the  limit problem.\\ 
A natural question that arises from Theorem \ref{i17} is the study of the shape of the bifurcation
diagram of the nonradial solutions.  From the classical bifurcation theory
we are not able to derive information of this type.
However, we conjecture that the nonradial solutions of Theorem \ref{i17} exist only when $\alpha$ is an even integer (see Figure 2).

\begin{figure}[h!] 
\centering
\includegraphics[scale=0.3]{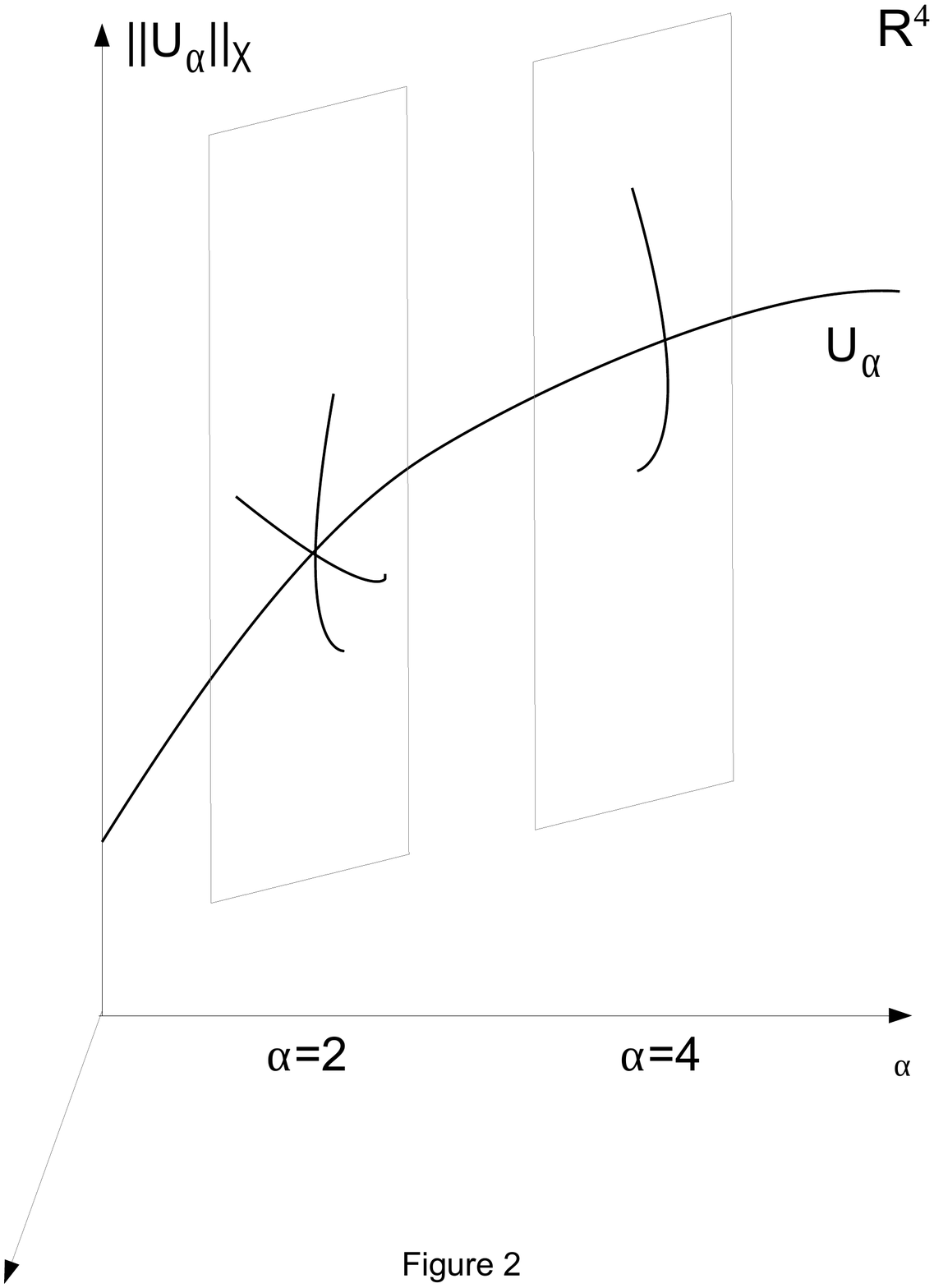}
\label{fig2}
\end{figure}

We have no proof of this, but we are going to make two remarks that support this conjecture.
The first one is the calculation of a branch of explicit solutions which bifurcate from $U_\a$, at least when $\a=2$ and $N$ is even.
\begin{proposition}\label{i18}
Let $\alpha=2$, $N \geq 4$ even and $x\in\R^{N}=\R^\frac N2\times\R^\frac N2$, $x=(x',x'')$ with  $x'\in\R^\frac N2$ and $x''\in\R^\frac N2$. Then, for any $a\in\R$, the functions
\begin{equation}\label{i19}
u(x) = u(|x'|,|x''|) = \frac1{(1+|x|^4 -2a(|x'|^2 - |x''|^2) + a^2)^\frac{N-2}4 }
\end{equation}
form a branch of solutions to \eqref{1} bifurcating from $U_2$.
\end{proposition}
The second reason that supports our conjecture is the following classification result for a  Liouville-type equations with singular data (see  J. Prajapat and G. Tarantello \cite{PT01}).
\begin{theorem}\label{i20}
Let us consider the problem
 \begin{equation}\label{i21}
\left\{\begin{aligned}
&-\Delta u=2(\alpha+2)^2|x|^\alpha e^u  \quad \hbox{in }\R^2, \\
&\int_{\R^2}|x|^\alpha e^u<+\infty.
\end{aligned}
\right.
\end{equation}
If $\alpha$ is not an {\em even} integer then the unique solutions to \eqref{i21} are given by
 \begin{equation}\label{i22}
u_\lambda(x)=\log\frac{\lambda^{\alpha+2}}{\left(1+\lambda^{\alpha+2}|x|^{\alpha+2}\right)^2}, \qquad \lambda > 0.
 \end{equation}
On the other hand, if $\alpha$ is an {\em even} integer we also have the following $nonradial$ solutions, for $a\in\R,\ \theta_0 \in[0,2\pi]$
 \begin{equation}\label{i23}
 v_{a,\theta_0}(x)=\log\frac1{\left(1+|x|^{\alpha+2}-2a|x|^\frac{\alpha+2}2\cos\left(\frac{\alpha+2}2(\theta-\theta_0)\right)+a^2\right)^2}.
 \end{equation}
 with $\theta$ the angle of $x$ in polar coordinates.
 \end{theorem}
Problem \eqref{i21}, that admits nonradial solutions only if $\a$ is an even integer,
can be seen as the equivalent of \eqref{1} if $N=2$.\\
Note also the similarities of \eqref{i23}, when $\alpha=2$,  with the explicit solution given by  \eqref{i19}.\\
Another important similarity between our results and those related to problem \eqref{i21},
 concerns the analogue of Theorem \ref{linearized} (\cite{DEM12}).
\begin{theorem} \label{i24}
Let $\a$ be an even integer and  $k=\frac{\a+2}2$. Let us consider the linearized problem associated to \eqref{i21}, i.e.,
 \begin{equation}\label{i25}
-\Delta z=2(\alpha+2)^2|x|^\alpha e^{u_1(x)}z  \quad \hbox{in }\R^2
 \end{equation}
Then the space of all bounded solutions of \eqref{i25} is spanned
by
 \begin{equation}\label{i26}
 Z_1(x)=\frac{1-|x|^{2+\a}}{1+|x|^{2+\a}},\quad Z_2(x)=\frac{P_1(x)}{1+|x|^{2+\a}},\quad Z_3(x)=\frac{P_2(x)}{1+|x|^{2+\a}}
 \end{equation}
 where $P_1(x)$ and $P_2(x)$ form a basis of $\mathbb{Y}_k(\R^2)$, the space of all homogeneous harmonic polynomials of degree $k$ in $\R^2$.
 \end{theorem}
We conclude showing a further application of Theorem \ref{linearized} that allows to give a
generalization of an existence result of single-peak
solutions for the almost critical H\'enon equation (see
\cite{GG12}) in bounded domains.
\begin{theorem} \label{i14}
Let $\alpha > 0$ different from an {\em even integer} and $\Om$
be a smooth bounded domain of $\R^N$ with $N\ge3$ and $0\in\Om$.
Then, for $\e$ small enough, there exists a solution $u_\e$ to
\begin{equation} \label{i15}
\left\{\begin{array}{ll}
-\Delta u=|x|^\alpha u^{\frac{N+2+2\alpha}{N-2}-\e}  & \hbox{ in }\Om\\
u>0 & \hbox{ in }\Om\\
u=0 & \hbox{ on }\de\Om.
\end{array}\right.
\end{equation}
\end{theorem}
This theorem was proved in \cite{GG12} under the more restrictive assumption
$0<\alpha\le1$, since it was not available the characterization of
the solutions of \eqref{1.4}. The proof is based on the Liapunov-Schmidt finite dimensional reduction method
and still works if $\a$ is  different from an {\em even integer} since the kernel of the linearized operator is one-dimensional.
The case of $\a$ even is much more difficult, due to the richness of the kernel of the linearized operator, and it seems difficult to handle as the previous one. Obviously, Theorem \ref{linearized} can be applied to asymptotic problems similar
to \eqref{i15}.

The paper is organized as follows. In Section \ref{s2} we prove Theorem \ref{linearized} and Corollary \ref{cor-1}. In Section \ref{s3} we study the approximated problem in the ball and we construct our approximated solution. In Section \ref{s6} we prove some estimates on the approximated solution which enable us to pass to the limit and in Section \ref{s7} we prove Theorem \ref{i17} and Proposition \ref{i18}. Finally in the Appendix we give a simplified proof of known results.

\sezione{The linearized operator } \label{s2}
In this section we consider the linearized problem  \eqref{i9} and
we give the proof of Theorem \ref{linearized} and of Corollary \ref{cor-1}.

\begin{proof}[Proof of Theorem \ref{linearized}]
We want to find solutions of
\begin{align}\label{1.7}
\begin{cases}
-\Delta V=C(\a)p_\a\frac{|x|^{\a}}{(1+|x|^{2+\a})^2}V  & \hbox{ in }\R^N\\
V\in D^{1,2}\left(\R^N\right)&
\end{cases}
\end{align}
of the form
\begin{equation*}
V(r,\theta) = \sum_{k=0}^{\infty} \psi_k(r)Y_k(\theta), \qquad \text{where} \qquad r=|x| \,\, , \,\, \theta=\frac{x}{|x|} \in S^{N-1}
\end{equation*}
and
\begin{equation*}
\psi_k(r) = \int_{S^{N-1}}V(r,\theta)Y_k(\theta)d\theta.
\end{equation*}
Here $Y_k(\theta)$ denotes the $k$-th spherical harmonics, i.e. it satisfies
\begin{equation*}
-\Delta_{S^{N-1}} Y_k = \mu_kY_k
\end{equation*}
where $\Delta_{S^{N-1}}$ is the Laplace-Beltrami operator on $S^{N-1}$ with the standard metric and $\mu_k$ is the $k$-th eigenvalue
of $-\Delta_{S^{N-1}}$. It is known that
$$\mu_k = k(N-2+k), \qquad k=0,1,2,\dots$$
whose multiplicity is
$$ \frac{(N+2k-2)(N+k-3)!}{(N-2)!\,k!} $$
and that
$$ Ker\left(\Delta_{S^{N-1}} + \mu_k\right) = \mathbb{Y}_k(\R^N)\left|_{S^{N-1}}\right. .$$

The function $V$ is a solution of \eqref{1.4} if and only if $\psi_k(r)$ satisfies
\begin{align} \label{1.8}
\begin{cases}
-\psi_k''(r) - \frac{N-1}{r}\psi_k'(r) + \frac{\mu_k}{r^2}\psi_k(r) = C(\a)p_{\a}\frac{r^\alpha}{(1+r^{2+\a})^2}\,\psi_k(r)\,, \quad \text{in} \,\,(0,\infty) \\
 \psi_k \in \mathcal{E} \\
 \psi_k'(0)=0 \,\,\, \text{if} \,\,\, k=0 \quad \text{and} \quad \psi_k(0)=0 \,\,\, \text{if} \,\,\, k \geq 1
\end{cases}
\end{align}
where
\begin{equation}\label{def-E}
 \mathcal{E} = \left\{ \psi \in C^1([0,\infty)) \quad \left| \quad \int_0^{\infty}r^{N-1}|\psi'(r)|^2dr< \infty \right. \right\}.
\end{equation}

We shall solve \eqref{1.8} using the following change of variables
\begin{equation} \label{1.8a}
 \eta_k(r) = \psi_k(r^{\frac{2}{2+\a}})
\end{equation}
that transforms \eqref{1.8} into the equation
\begin{align} \label{1.9}
\begin{cases}
 -\eta_k''(r) - \frac{M-1}{r}\eta_k'(r) + \frac{4\mu_k}{(2+\a)^2}\frac{\eta_k(r)}{r^2} =
M(M+2)\,\frac{\eta_k(r)}{(1+r^2)^2}\,, \quad \text{in}\,\, (0,\infty) \\
\eta_k \in \widetilde{\mathcal{E}} \\
 \eta_k'(0)=0 \,\,\, \text{if} \,\,\, k=0 \quad \text{and} \quad \eta_k(0)=0 \,\,\, \text{if} \,\,\, k \geq 1
\end{cases}
\end{align}
where

\begin{equation} \label{1.8ab}
M=\frac{2(N+\a)}{2+\a}\,, \qquad  \widetilde{\mathcal{E}} = \left\{ \eta \in C^1([0,\infty)) \quad \left| \quad \int_0^{\infty}r^{M-1}|\eta'(r)|^2dr< \infty \right. \right\}.
\end{equation}

Note that we have $\eta_0'(0)=0$ since $\eta_0'(r)=\frac2{\alpha+2}r^{-\frac\a{\a+2}}\psi_0'\left(r^\frac2{\alpha+2}\right)$ and by \eqref{1.8} $\psi_0'(r) =O(r^{\a+1})$ near $r=0$.

Fixed $M$ let us now consider the following eigenvalue problem
\begin{equation} \label{1.10}
-\eta''(r) - \frac{M-1}{r}\eta'(r) + \beta \frac{\eta(r)}{r^2} =
M(M+2)\,\frac{\eta(r)}{(1+r^2)^2}\,, \quad \text{in}\,\, (0,\infty).
\end{equation}
The equation \eqref{1.10} is a singular Sturm-Liouville problem, it has a sequence of simple eigenvalues $\beta_1 > \beta_2 > \dots$ and if $\eta_j$
 is an eigenfunction of $\beta_j$ then $\eta_j$ has exactly $j-1$ zeros in $(0,\infty)$, see for example \cite[Theorem 10.12.1]{Z05}.

When $M$ is an integer we can study \eqref{1.10} as the linearized operator of the equation $-\Delta U = U^{\frac{M+2}{M-2}}$ around the
standard solution $U(x) = \frac{1}{(1+|x|^2)^{\frac{M-2}{2}}}\,$, (note that we always have $M>2$). In this case, we know that
\begin{equation} \label{1.11}
 \beta_1 = M-1\,;\,\, \beta_2=0 \quad \text{and} \quad \eta_1(r)=\frac{r}{(1+r^2)^{\frac{M}{2}}}\, ; \,\, \eta_2(r) = \frac{1-r^2}{(1+r^2)^{\frac{M}{2}}}\, .
\end{equation}
However, even when $M$ is not an integer we readily see that \eqref{1.11} remains true. Therefore, we can conclude that \eqref{1.9} has nontrivial
solutions if and only if
$$ \frac{4\mu_k}{(2+\a)^2} \in \{0\,, M-1\},$$
which means that
$$ k=0 \quad \text{or} \quad \a=2(k-1). $$
Turning back to \eqref{1.8} we obtain the solutions
\begin{align}
\psi_0(r) = \frac{1-r^{2+\a}}{(1+r^{2+\a})^{\frac{N+\a}{2+\a}}} & \qquad  \text{if} \,\,\, \a \neq 2(k-1) \,\,,\, k \in \N \label{1.11a}\\
\psi_0(r) = \frac{1-r^{2+\a}}{(1+r^{2+\a})^{\frac{N+\a}{2+\a}}} \,\,,\,\, &\psi_k(r) = \frac{r^k}{(1+r^{2+\a})^{\frac{N+\a}{2+\a}}} \label{1.11b}
\qquad  \text{if} \,\,\, \a = 2(k-1) \,\,,\, k \in \N
\end{align}
and the proof of the theorem is complete.
\end{proof}

%
\begin{proof}[Proof of Corollary \ref{cor-1}]
Let us now consider the following eigenvalue problem,
\begin{align}\label{1.12}
\begin{cases}
-\Delta V=\Lambda C(\a)p_\a\frac{|x|^{\a}}{(1+|x|^{2+\a})^2}V  & \hbox{ in }\R^N\\
V\in D^{1,2}\left(\R^N\right).&
\end{cases}
\end{align}
The Morse index of $U_{\a}$ is the sum of the dimensions of the eigenspaces of \eqref{1.12} related to $\Lambda < 1$.

As in the  proof of Theorem \ref{linearized}, we are led to the equation
\begin{align} \label{1.13}
\begin{cases}
-\psi_k''(r) - \frac{N-1}{r}\psi_k'(r) + \frac{\mu_k}{r^2}\psi_k(r) = \Lambda C(\a)p_{\a}\frac{r^\alpha}{(1+r^{2+\a})^2}\,\psi_k(r)\,, \quad \text{in} \,\,(0,\infty) \\
 \psi_k \in \mathcal{E}, \quad \Lambda < 1 \\
 \psi_k'(0)=0 \,\,\, \text{if} \,\,\, k=0 \quad \text{and} \quad \psi_k(0)=0 \,\,\, \text{if} \,\,\, k \geq 1.
\end{cases}
\end{align}
For every $k\geq 0$, the problem \eqref{1.13} has an increasing sequence of eigenvalues $\Lambda_{j,k}$, $j=1,2,\dots$ and an associated eigenfunction
has exactly $j-1$ zeros in $(0,\infty)$.

We claim that $\Lambda_{j,k}\geq 1$ for all $j \geq 2$ and any $k \geq 0$. Indeed, using the transformation \eqref{1.8a}, as in the previous theorem we get
\begin{align} \label{1.14}
\begin{cases}
 -\eta_k''(r) - \frac{M-1}{r}\eta_k'(r) + \frac{4\mu_k}{(2+\a)^2}\frac{\eta_k(r)}{r^2} =
\Lambda M(M+2)\,\frac{\eta_k(r)}{(1+r^2)^2}\,, \quad \text{in}\,\, (0,\infty) \\
\eta_k \in \widetilde{\mathcal{E}} \\
 \eta_k'(0)=0 \,\,\, \text{if} \,\,\, k=0 \quad \text{and} \quad \eta_k(0)=0 \,\,\, \text{if} \,\,\, k \geq 1
\end{cases}
\end{align}
and, for any fixed $\Lambda < 1$, we consider the more general equation
\begin{align} \label{1.15}
 -\eta''(r) - \frac{M-1}{r}\eta'(r) + \b\frac{\eta(r)}{r^2} =
\Lambda M(M+2)\,\frac{\eta(r)}{(1+r^2)^2}\,, \quad \text{in}\,\, (0,\infty)
\end{align}
which has a sequence of eigenvalues $\{\b_1(\Lambda) > \b_2(\Lambda) > \dots \}$. Since $\Lambda <1$, using the min-max characterization of the eigenvalues
and \eqref{1.11} we infer that
\begin{equation}\label{1.16}
M-1 = \b_1(1) > \b_1(\Lambda) \qquad \text{and} \qquad  0 = \b_2(1) > \b_2(\Lambda).
\end{equation}
Now, let $\psi_{j,k}$ be an eigenfunction of \eqref{1.13} related to $\Lambda_{j,k}$ for some $j \geq 2$, $k \geq 0$ and suppose that $\Lambda_{j,k} < 1$. Then
$\psi_{j,k}(r^{\frac{2}{2 + \a}})$ is an eigenfunction of \eqref{1.14} which changes sign related to $\beta(\Lambda_{j,k}) = \frac{4\mu_k}{(2+\a)^2} \geq 0$. On the other hand, by \eqref{1.16} with $\Lambda = \Lambda_{j,k} < 1$ we have that $\beta(\Lambda_{j,k})\leq \b_2(\L) < 0$, and a contradiction arises. Hence the claim is proved.

By the previous claim it follows that $j=1$ and
therefore only the first eigenvalues of \eqref{1.13} may contribute to the Morse index. When $\a$ is an even integer, from \eqref{1.11b}, we already know that
$$ \psi_k(r) = \frac{r^k}{(1+r^{2+\a})^{\frac{N+\a}{2+\a}}} = \frac{r^k}{(1+r^{2+\a})^{\frac{N+2(k-1)}{2+\a}}} $$
is the first eigenfunction of \eqref{1.13}, for $k=\frac{\a+2}{2}$, with eigenvalue $\Lambda = 1$. However, for general $\a>0$ and $k \geq 0$, we can check that the first eigenfunction of \eqref{1.13} is still
\begin{equation} \label{1.17}
\psi_{1,k}(r) = \frac{r^k}{(1+r^{2+\a})^{\frac{N+2(k-1)}{2+\a}}}
\end{equation}
and the first eigenvalue of \eqref{1.13} is
\begin{equation}
\label{2.16}
\Lambda_{1,k} = \frac{(N-2+2k)(N+\a+2k)}{(N+2+2\a)(N+\a)}.
\end{equation}
A straightforward computation shows that
\begin{equation*}
\Lambda_{1,k} < 1 \quad \Leftrightarrow \quad k < \frac{\a+2}{2}.
\end{equation*}
Therefore, the eigenvalues of \eqref{1.12} that satisfy $\Lambda < 1$ are precisely $\Lambda_{1,k}$ for $k < \frac{\a+2}{2}$ and the eigenfunctions of \eqref{1.12} related to $\Lambda_{1,k}$ are linear combinations of functions of the form
\begin{equation}
\label{1.18}
V=\frac{|x|^k}{(1+|x|^{2+\a})^{\frac{N+2(k-1)}{2+\a}}} Y_k\left(\frac{x}{|x|}\right) \qquad \text{where}  \quad Y_k \in \mathbb{Y}_k(\R^N).
\end{equation}
Since dim$(\mathbb{Y}_k(\R^N)) = \frac{(N+2k-2)(N+k-3)!}{(N-2)!\,k!}$ the proof is now complete.
\end{proof}

\sezione{The approximated problem}\label{s3}
In this section we consider the problem
\begin{equation}
\label{p-epsilon}
\left\{
\begin{array}{ll}
-\Delta u= C(\alpha)|x|^{\a}\left(u+U_{\a}(\frac 1 \e)\right)^{p_{\a}},&\text{ in } B_{\frac 1{\e}}(0), \\
u\geq 0 &\text{ in } B_{\frac 1{\e}}(0), \\
u=0,&\text{ on } \partial B_{\frac 1{\e}}(0) ,
\end{array}
\right.
\end{equation}
where $\a>0$ is fixed and $B_{\frac 1{\e}}(0)$ denotes
the ball of radius $\frac 1{\e}$ centered at the origin.
We let
\begin{equation}\label{w-epsilon}
u_{\e,\a}(x)
=U_{\a}(x)-U_{\a}\left(\frac 1 \e\right)=\frac 1{\left(1+|x|^{2+\a}\right)^{\frac{N-2}{2+\a}}}-\frac {\e^{N-2}}{\left(1+\e^{2+\a}\right)^{\frac{N-2}{2+\a}}}
\end{equation}
where $U_{\a}(x)$ is as defined in \eqref{i9}. We have
\begin{lemma}\label{lemma-radial-nondegeneracy}
For any $\a\geq 0$ and for any $0<\e<1$ sufficiently small, the function $u_{\e,\a}$ is a radial solution of \eqref{p-epsilon} which is nondegenerate in the space of the radial functions.
\end{lemma}
\begin{proof}
For any $\e>0$ and for any $x\in B_{\frac 1{\e}}(0)$  it is immediate to check that $u_{\e,\a}$ is a radial solution to \eqref{p-epsilon}.
The function $u_{\e,\a}$ is radially nondegenerate if the linearized problem
\begin{equation}
\label{L-epsilon}
\left\{
\begin{array}{ll}
-\Delta \varphi= p_{\a}C(\alpha)|x|^{\a}\left(u_{\e,\a}+U_{\a}(\frac 1 \e)\right)^{p_{\a}-1}\varphi,&\text{ in } B_{\frac 1{\e}}(0), \\
\varphi=0,&\text{ on } \partial B_{\frac 1{\e}}(0) ,
\end{array}
\right.
\end{equation}
does not have radial nontrivial solutions. Using \eqref{w-epsilon} we can rewrite \eqref{L-epsilon} in radial coordinates, i.e.
\begin{equation}
\label{L-rad-epsilon}
\left\{
\begin{array}{ll}
- (r^{N-1}\varphi')'= p_{\a}C(\alpha)\frac {r^{\a+N-1}}{\left(1+r^{2+\a}\right)^2}\, \varphi,&\text{ in } (0,\frac 1{\e}), \\
\varphi'(0)=0,\quad \varphi(\frac 1{\e})=0.
\end{array}
\right.
\end{equation}
Observe that the function $z(r):=\frac{1-r^{2+\a}}{\left(1+r^{2+\a}\right)^{\frac{N+\a}{2+\a}}}$ satisfies the linearized equation
\begin{equation}\label{L-rad}
- (r^{N-1}z')'= p_{\a}C(\alpha)\frac {r^{\a+N-1}}{\left(1+r^{2+\a}\right)^2}\, z,\text{ in } \R_+, \qquad z'(0)=0
\end{equation}
but $z(\frac 1{\e}) \neq 0$, for $\e < 1$.

Multiplying \eqref{L-rad-epsilon} by $z$, \eqref{L-rad} by $\varphi$ and integrating we get
$$ \varphi'\left(\frac 1{\e}\right) = 0$$
and since $\varphi(\frac 1{\e}) = 0$ we must have $\varphi \equiv 0$.
 This implies that \eqref{L-rad-epsilon} does not have any nontrivial solution and hence $u_{\e,\a}$ is radially nondegenerate.
\end{proof}

\noindent Let $\beta>0$ be such that $\frac{N}{N+2}(N-2)< \b <N-2$.
For every $g\in L^{\infty}(\R^N)$, we define the weighted norm
\begin{equation}
\label{norm-beta}
\nor g\nor_{\b}:=\sup_{x\in \R^N}(1+|x|)^{\b}|g(x)|.
\end{equation}
Set $L^{\infty}_{\b}(\R^N):=\{g\in L^{\infty}(\R^N) \text{ such that }\exists\, C>0 \text{ and }\nor g\nor_{\b}<C\}$ and
\begin{equation}\label{X}
X=D^{1,2}(\R^N)\cap L^{\infty}_{\b}(\R^N).
\end{equation}
$X$ is a Banach space with the norm
\begin{equation}\label{norm-x}
\nor g\nor_X:=\max\{\nor g\nor_{1,2},\nor g\nor_{\b}\}
\end{equation}
where $\nor \cdot \nor_{1,2}$ denotes the usual norm in $D^{1,2}(\R^N)$, i.e. $\nor g\nor_{1,2}=\left( \int_{\R^N}|\na g|^2\, dx\right)^{\frac 12}$ for $g\in D^{1,2}(\R^N)$.
\noindent Now  we are in position to  state the following
\begin{proposition}\label{p3.2}
Let $\e_n$ and $\a_n$ be sequences such that $\e_n\to 0$ and $\a_n\to  \a>0$ as $n\to +\infty$. Denoting by $u_{n}:=u_{\e_n,\a_n}$, with $ u_{\e,\a}$ as defined in \eqref{w-epsilon}, we have that
\begin{equation}\label{convergence}
\nor u_{n}-U_{\a}\nor _X\to 0 \qquad \text{as} \,\,\, n \to + \infty
\end{equation}
where $u_{n}$ is assumed to be extended by zero outside $B_{\frac{1}{\e_n}}$.
\end{proposition}
\begin{proof}
By \eqref{norm-x} we need to prove that $\nor u_{n}-U_{\a}\nor_{1,2}\to 0$  and $\nor u_{n}-U_{\a}\nor_{\b}\to 0$ as $n\to +\infty$. By the definition of $u_{n}$, letting $B_n:=B_{\frac{1}{\e_n}}$, we have
$$\int_{\R^N}|\na u_{n}-\na U_{\a}|^2\, dx=\int_{B_n}|\na U_{\a_n} - \na U_{\a}|^2\, dx + \int_{\R^N\setminus B_n}|\na U_{\a}|^2\, dx .$$
Since $U_{\a}\in D^{1,2}(\R^N)$ it follows that
$$\int_{\R^N\setminus B_n}|\na U_{\a}|^2\, dx \to 0 \qquad \text{as} \,\,\, n \to +\infty  $$
and by the decay of $|\na U_{\a}|$
$$ \int_{B_n}|\na U_{\a_n} - \na U_{\a}|^2\, dx \to 0 \qquad \text{as} \,\,\, n \to +\infty .$$
Finally, using the definitions of $u_{n}$, $U_{\a}$ and the mean value Theorem we have
\begin{align*}
\nor u_{n}-U_{\a}\nor_{\b}&=\sup_{x\in \R^N }(1+|x|)^{\b}|u_{n}(x)-U_{\a}(x)|\\
\begin{split}
 &\leq \sup_{x\in B_n} (1+|x|)^{\b}\left| U_{\a_n}(x) - U_{\a}(x) \right| +\sup_{x\in B_n} (1+|x|)^{\b}U_{\a_n}\left(\frac 1 {\e_n}\right) \\
 &\qquad +\sup _{x\in \R^N\setminus B_n}\frac{(1+|x|)^{\b}} {\left(1+|x|^{2+\a}\right)^{\frac{N-2}{2+\a}}}
\end{split} \\
&\leq O(|\a-\a_n|) + O(\e_n^{N-2-\b})
\end{align*}
and since $\b<N-2$, the proposition follows.
\end{proof}

\subsection{Convergence of the spectrum}\label{s4}
Let us consider the linearized problem associated with \eqref{p-epsilon}, i.e.
\begin{equation}\label{aut-epsilon}
\left\{\begin{array}{ll}
-\Delta v=p_{\a}C(\a)|x|^{\a}\left(u_{\e,\a}+U_{\a}(\frac 1 \e)\right)^{p_{\a}-1}v,&\text{ in } B_{\frac 1{\e}}(0), \\
v=0,&\text{ on } \partial B_{\frac 1{\e}}(0).
\end{array}
\right.
\end{equation}
Recalling that $u_{\e,\a}=U_{\a}-U_{\a}(\frac 1 \e)$ we can
rewrite \eqref{aut-epsilon} in the following way,
\begin{equation}\label{aut}
\left\{\begin{array}{ll}
-\Delta v=p_{\a}C(\a)\frac{|x|^{\a}}{\left(1+|x|^{2+\a}\right)^2}v,&\text{ in } B_{\frac 1{\e}}(0), \\
v=0,&\text{ on } \partial B_{\frac 1{\e}}(0).
\end{array}
\right.
\end{equation}
We can decompose problem \eqref{aut} in radial part and angular
part using the spherical harmonic functions (as in Section
\ref{s2}), getting that $v$ is a solution of \eqref{aut} if and
only if $v_k(r):=\int_{S^{N-1}}v(r,\theta)Y_k(\theta)\, d\theta$
is a solution of
\begin{equation}\label{a-epsilon}
\left\{\begin{array}{ll} -v_k''-\frac {N-1}r
v'_k+\frac{\mu_k}{r^2}
v_k=p_{\a}C(\a)\frac{r^{\a}}{\left(1+r^{2+\a}\right)^2}v_k
,&\text{ in } (0,\frac 1{\e}), \\
v_k'(0)=0=v_k(\frac 1{\e}) \,\,\, \text{if} \,\,\, k=0 \quad
\text{and} \quad v_k(0)=0=v_k(\frac 1{\e}) \,\,\, \text{if} \,\,\,
k \geq 1
\end{array}
\right.
\end{equation}
for some $\mu_k=k(N+k-2)$, where $Y_k(\theta)$ denotes the $k$-th
spherical harmonic function. In this way we have that all the
eigenfunctions of \eqref{aut-epsilon} are given by
$v_k(r)Y_k(\theta)$, if $v_k$ is a solution of
\eqref{a-epsilon}.\\
By Lemma \ref{lemma-radial-nondegeneracy} we get that \eqref{a-epsilon} admits a solution only if $k \neq 0$.\\
Let us introduce the following eigenvalue problem,
\begin{equation}\label{b3-epsilon}
\left\{\begin{array}{ll}
-z''-\frac {N-1}r z'-p_{\a}C(\a)\frac{r^{\a}}{\left(1+r^{2+\a}\right)^2}z=\L\frac z{r^2},&\text{ in } (0,\frac 1{\e}), \\
z(\frac 1{\e})=0\quad z\in L^{\infty}(0,\frac 1{\e})
\end{array}
\right.
\end{equation}
which admits a sequence of eigenvalues $\L_h^{\e}(\a)$ which are simple.
Observe that if $z\in L^{\infty}(0,\frac 1{\e})$ is an eigenfunction of \eqref{b3-epsilon} with $\L\ne0$, then $z(0)=0$. \\
Then \eqref{a-epsilon} is equivalent to find $\a>0$ and integers $h,k\ge1$ such that
\begin{equation}\label{c-epsilon}
-\mu_k=\L_h^{\e}(\a)
\end{equation}
We have the following lemma,
\begin{lemma}\label{l4.0}
We have that, for $0<\e<1$,
\begin{equation}\label{l4.1}
\L_1^{\e}(\a)<0
\end{equation}
and
\begin{equation}\label{l4.2}
\L_2^{\e}(\a)>0.
\end{equation}
\end{lemma}
\begin{proof}
Since $Z(r)=\frac{1-r^{2+\a}}{(1+r^{2+\a})^{\frac{N+\a}{2+\a}}}$ is a positive solution to \eqref{a-epsilon} with $\e=1$, we get that $\L_1^{1}(\a)=0$. 
By the monotonicity of the first eigenvalue with respect to $\e$ we get that $\L_1^{\e}(\a)<\L_1^{1}(\a)=0$  and \eqref{l4.1} follows.\\
Concerning \eqref{l4.2}, it follows from the monotonicity of the
eigenvalues with respect to the domain. Let us denote by
$\L=\L_h(\a)$ the eigenvalues of the problem,
\begin{equation}\label{l4.3}
\left\{\begin{array}{ll}
-z''-\frac {N-1}r z'-p_{\a}C(\a)\frac{r^{\a}}{\left(1+r^{2+\a}\right)^2}z=\L\frac z{r^2},&\text{ in } (0,+\infty), \\
z\in\mathcal{E}
\end{array}
\right.
\end{equation}
with $\mathcal{E}$ as in \eqref{def-E}. We get that
$\L_2^{\e}(\a)>\L_2(\a)$. Using \eqref{1.8a}, \eqref{1.15} and
\eqref{1.16} we derive that $\L_2(\a)=0$ which gives the claim.
\end{proof}
By the previous lemma, since $\mu_k>0$, we get that the equation
\eqref{c-epsilon} has solutions only if $h=1$.\\
Let us denote by $\L_1(\a)$ and $z_1$ the first eigenvalue and the
first eigenfunction of problem \eqref{l4.3}. Next lemma studies
the convergence of $\L_1^{\e}(\a)$ as $\e\longrightarrow0$.
\begin{lemma}\label{l4.4}
Let $\e_n$ be a sequence such that $\e_n\to 0$ as $n\to +\infty$.
Let $\L_1^n(\a)$ denote the first eigenvalue of \eqref{b3-epsilon}
in $(0,\frac 1{\e_n})$ (related to the exponent $\a$). Then
\begin{equation}\label{conv-autov}
\L_1^n(\a)\to \L_1(\a)=-\frac{(2N+\a-2)(\a+2)}4\quad \text{ as
}n\to +\infty.
\end{equation}
Moreover the convergence in \eqref{conv-autov} is uniform in $\a$
on compact sets of $(0,+\infty)$.\\
Finally, denoting by $z_n(r)$ the first positive eigenfunction of
\eqref{b3-epsilon} related to $\L_1^n(\a)$ with
$||z_n||_{\infty}=1$, we have that
\begin{equation}\label{conv-autov2}
|\left(z_n(r)\right)'|\leq Cr^{1-N}\quad \quad |z_n(r)|\leq Cr^{2-N},
\end{equation}
\begin{equation}\label{conv-autov3}
z_n\rightarrow
z(r)=\frac{r^\frac{2+\a}2}{\left(1+r^{2+\a}\right)^\frac{N-2}{2+\a}}\quad\hbox{uniformly
on compact sets of }[0, +\infty).
\end{equation}
\end{lemma}
\begin{proof}
The convergence in \eqref{conv-autov} follows from the dependence on the domain of
the eigenvalues or from the Sturm-Liouville theory (see \cite[Theorems 5.3 and 6.4]{BEWZ93}
for example).\\
Moreover, by \eqref{1.8}-\eqref{1.8ab}, it follows that
$\L_1(\a)=-(M-1)\frac{(\a+2)^2}4=-\frac{(2N+\a-2)(\a+2)}4$.\\
 Let $z_n(r)$ be the first positive eigenfunction of \eqref{b3-epsilon}
related to $\L_1^n(\a)$ and normalized with the $L^{\infty}$-norm.
So $z_n$ satisfies
\begin{equation}\label{b4-epsilon}
\left\{\begin{array}{ll}
-z_n''-\frac {N-1}r z_n'-p_{\a}C(\a)\frac{r^{\a}}{\left(1+r^{2+\a}\right)^2}z_n=
\L_1^n(\a)\frac{z_n}{r^2},&\text{ in } (0,\frac 1{\e_n}), \\
z_n(\frac 1{\e_n})=0,\quad ||z_n||_\infty=1
\end{array}
\right.
\end{equation}
Moreover we have that, for $r$ large enough and $\a>0$,
\begin{equation}\label{d1}
p_{\a}C(\a)\frac{r^{N-1+\a}}
{\left(1+r^{2+\a}\right)^2}+\L_1^n(\a)r^{N-3}<0,
\end{equation}
where we used that $\L_1^n(\a)<-C$ for $n$ large.\\
Integrating \eqref{b4-epsilon} on $(r,\frac 1{\e_n})$ then we get
\begin{equation}\nonumber
r^{N-1}z_n'(r)=\left(\frac 1{\e_n}\right)^{N-1} z_n'\left(\frac 1{\e_n}\right)+p_{\a}C(\a)\int_{r}^{\frac 1{\e_n}}
\frac{s^{N-1+\a}}
{\left(1+s^{2+\a}\right)^2}z_n(s)ds +\L_1^n(\a)\int_{r}^{\frac 1{\e_n}} s^{N-3}z_n(s)\,
ds
\end{equation}
and since $0\le z_n(r)\leq 1$ and $z_n'\left(\frac 1{\e_n}\right)<0$ this implies
\begin{equation}\label{d2}
z'(r)<0\quad\hbox{for }r\hbox{ large enough}.
\end{equation}
Integrating \eqref{b4-epsilon} on $(0,r)$  we get
\begin{equation}\label{l-2}
-r^{N-1}z_n'(r)= p_{\a}C(\a)\int_0^{r}
\frac{s^{N-1+\a}}
{\left(1+s^{2+\a}\right)^2}z_n(s)+\L_1^n(\a)\int_0^{r}s^{N-3}z_n(s)\,
ds.
\end{equation}
Observe that since $\L_1^n(\a)<0$ and $0\le z_n(r)\leq 1$ we get, using \eqref{d2}
\begin{equation}\label{l-3}
\left|z_n'(r)\right|\leq \frac{C}{r^{N-1}}\left\{
\begin{array}{ll}
C & \text{ if }N<4+\a\\
C+\log r & \text{ if }N=4+\a\\
C+r^{N-4-\a} & \text{ if }N>4+\a
\end{array}\right.
\end{equation}
If $N<4+\a$ we get the optimal decay for $z_n(r)$ and $\left(z_n(r)\right)'$, i.e.
\begin{equation}\label{l-4}
\left|z_n'(r)\right|\leq Cr^{1-N}\quad \quad z_n(r)\leq Cr^{2-N}
\end{equation}
if $r$ is large enough. If else $N\ge4+\a$, inserting \eqref{l-3} into \eqref{l-2} and iterating the previous procedure,
after a finite number of steps we get \eqref{l-4} for any $n$ and
for any $\a$ on compact sets of $(0,+\infty)$.
This shows \eqref{conv-autov2}.\\
 From \eqref{b3-epsilon} and \eqref{conv-autov2} we have that
$$\int_0^{+\infty}r^{N-1}|\left(z_n(r)\right)'|^2 \, dr\leq C$$
(where $z_n$ is assumed to be zero for $r>\frac 1{\e_n}$) so that
$z_n\to z$ in $\mathcal{E}$ (weakly), a.e. in $(0,+\infty)$ and
uniformly on compact sets of $[0,+\infty)$. Here $z$ is the first
positive eigenfunction of \eqref{l4.3}. Using \eqref{conv-autov2}
again, we can pass to the limit into \eqref{b3-epsilon} getting
that $z$ is a solution of \eqref{l4.3} corresponding to the
eigenvalue $\L_1(\a)$. Moreover $z\not\equiv 0$ since, from
\eqref{conv-autov2}, the maximum point of $|z_n(r)|$ converges to
a
point $r_0\in [0,+\infty)$ and $|z(r_0)|=1$ from the uniform convergence.\\
We finish the proof by proving the uniform convergence of
$\L_1^n(\a)$ to $\L_1(\a)$ on compact sets. Let us multiply
\eqref{b3-epsilon} by $z$ and we integrate on $(0,\frac1{\e_n})$,
we  multiply \eqref{l4.3} by $ z_n$ and we integrate on
$(0,\frac1{\e_n})$, then we subtract getting
\begin{align*}
&-\left( \frac1{\e_n}\right)^{N-1}(z_n)'\left( \frac1{\e_n}\right)z \left( \frac1{\e_n}\right)
\\ &=\left[\L_1^n(\a)- \L_1(\a)\right]\int_0^{\frac1{\e_n}} r^{N-3}z_n(r)z(r)\, dr.
\end{align*}
Then
$$\left( \frac1{\e_n}\right)^{N-1}(z_n)'\left( \frac1{\e_n}\right) z\left( \frac1{\e_n}\right)=O\left(\e_n^{N-2}\right)$$
as $n\to +\infty$, uniformly in $\a$ on compact sets of $(0,+\infty)$, while
$$\int_0^{\frac1{\e_n}} r^{N-3}z_n(r)z(r)\, dr\rightarrow\int_0^{+\infty}r^{N-3}z^2(r)\, dr>0$$
as $n\to +\infty$, uniformly in $\a$ on compact sets of $(0,+\infty)$. This implies that
$$\sup_{\a\in K}\left|\L_1^n(\a)- \L_1(\a)\right|=o(1)$$
as $n\to +\infty$, for any compact set $K\subset (0,+\infty)$ and this concludes the proof.
\end{proof}
\begin{proposition}
Let $\e_n$ be a sequence such that $\e_n\to 0$ as $n\to +\infty$
and $\L_1^n(\a)$ denotes the first eigenvalue of
\eqref{b3-epsilon} in $(0,\frac1{\e_n})$. Then we have that,
\begin{equation}\label{c1}
\frac{\partial \L_1^n(\a)}{\partial\a}\rightarrow-\frac{N+\a}2<0\quad\hbox{uniformly on the compact set }K\subset(0,+\infty).
\end{equation}
Moreover, for any integer $k\ge1$, the equation
\begin{equation}\label{c2-epsilon}
-k(N+k-2)=-\mu_k=\L_1^{n}(\a)
\end{equation}
admits exactly one solution $\a_k^n$ and
\begin{equation}\label{c2}
\a_k^n\rightarrow2(k-1)\quad\hbox{as } n\rightarrow \infty.
\end{equation}
\end{proposition}
\begin{proof}
By known results (see \cite{K95} for example) we have that if
$z_n$ is the first eigenfunction in \eqref{b3-epsilon} then
$\frac{\partial z_n}{\partial\a}$ and $\frac{\partial
\L_1^n(\a)}{\partial\a}$ are continuous functions. Deriving
\eqref{b3-epsilon} with respect to $\a$ we get that
$\frac{\partial z_n}{\partial\a}$ solves
\begin{equation}\label{b2-epsilon}
\left\{\begin{array}{ll}
-\left(\frac{\partial z_n}{\partial\a}\right)''-\frac {N-1}r \left(\frac{\partial z_n}{\partial\a}\right)'-p_{\a}C(\a)\frac{r^{\a}}{\left(1+r^{2+\a}\right)^2}\frac{\partial z_n}{\partial\a}-
\frac{\partial }{\partial\a}\left(p_{\a}C(\a)\frac{r^{\a}}{\left(1+r^{2+\a}\right)^2}\right) z_n\\
=\frac{\partial \L_1^n(\a)}{\partial\a}\frac { z_n}{r^2}+ \L_1^n(\a)\frac {\frac{\partial z_n}{\partial\a}}{r^2}\qquad\text{ in } (0,\frac 1{\e_n}), \\
\left(\frac{\partial z_n}{\partial\a}\right)(\frac 1{\e_n})=0
\end{array}
\right.
\end{equation}
Multiplying \eqref{b3-epsilon} by $\frac{\partial
z_n}{\partial\a}$ and \eqref{b2-epsilon} by $ z_n$, integrating
and subtracting we get
\begin{equation}\label{c3}
-\int\limits_0^\frac1{\e_n}\frac{\partial }{\partial\a}\left(p_{\a}C(\a)\frac{r^{\a}}{\left(1+r^{2+\a}\right)^2}\right) z_n^2(r)r^{N-1}dr=
\frac{\partial \L_1^n(\a)}{\partial\a}\int\limits_0^\frac1{\e_n}z_n^2(r)r^{N-3}dr
\end{equation}
By Lemma \ref{l4.4} we can pass to the limit in \eqref{c3} and we get
\begin{equation}\label{c4}
\frac{\partial \L_1^n(\a)}{\partial\a}\longrightarrow-\frac{\int\limits_0^{+\infty}\frac{\partial }{\partial\a}\left(p_{\a}C(\a)\frac{r^{\a}}{\left(1+r^{2+\a}\right)^2}\right) z^2(r)r^{N-1}dr}{\int\limits_0^{{+\infty}}z^2(r)r^{N-3}dr}=-\frac{N+\a}2
\end{equation}
uniformly on the compact sets of $[0,\infty)$. This proves \eqref{c1}.\\
Finally, since $\frac{\partial \L_1^n(\a)}{\partial\a}<0$ and
$\L_1^n(\a)\longrightarrow \L_1(\a)$, we have that the equation
\eqref{c2-epsilon} has exactly one solution $\a_k^n$ for any
$k\ge1$. Since $\L_1(\a)=-\frac{(2N+\a-2)(\a+2)}4$ we derive that
$\a_k=\lim\limits_{n\rightarrow \infty}\a_k^n$ solves
\begin{equation}\label{c5}
\frac{(2N+\a_k-2)(\a_k+2)}4=\mu_k=k(N-2+k)
\end{equation}
and then $\a_k=2(k-1)$ which ends the proof.
\end{proof}

\begin{remark}\label{rem-2}\rm
Let $n$ be fixed. From \eqref{a-epsilon}, \eqref{c-epsilon} and \eqref{c2-epsilon} we have that the solution $u_{n,\a}$ of \eqref{p-epsilon} corresponding to $\e_n$ is degenerate if and only if $\a=\a_k^n$ for $k=1,2,\dots$. Moreover, at these points $\a_k^n$ the Morse index of the solution $u_{n,\a_k^n}$ changes.\\
In particular, passing from $\a_k^n-\d$ to $\a_k^n+\d$, for $\d$
small enough, the Morse index of $u_{n,\a}$ increases of the
dimension of the eigenspace ${Ker}(\Delta_{S^{N-1}}+\mu_k)$.
\end{remark}

\subsection{The bifurcation result in the ball}\label{s5}
\noindent To state the bifurcation result we need some notations.
As before we
denote by $u_{n,\a}$ the radial solution of \eqref{p-epsilon} corresponding to the exponent $\a$, for  $\e=\e_n$, and by $B_n$ the ball of radius $\frac 1{\e_n}$ centered at the origin. We will denote by
\begin{equation}\label{S}
\cSn:=
\left\{
\begin{split}
(\a,u_{n,\a})\in  (0,+\infty)\times C^{1,\g}_0(\overline B_n)\, \hbox{ such that } u_{n,\a}
    \\ \hbox{ is the radial}
\hbox{ solution of (\ref{p-epsilon}) defined in \eqref{w-epsilon}}
\end{split}
\right\}
\end{equation}
Let us recall that, given the curve $\cSn$,
a point $(\a_i,u_{n,\a_i})\in \cSn$
is a {\em  nonradial bifurcation point } if in  every neighborhood of
$(\a_i,u_{n,\a_i})$  in $(0,+\infty)\times C^{1,\g}_0(\overline B_n)$
there exists a point $(\a,v_{n,\a})$ such that $v_{n,\a}$ is a nonradial solution of (\ref{p-epsilon}) in $B_n$.
We are in position to state our first result.
\begin{theorem}\label{t-bif}
Let us fix $n\in \N$ and let $\a_k^n$ be as defined in \eqref{c2-epsilon}. Then the points $(\a_k^n, u_{n,\a_k^n})$ are nonradial bifurcation points for
the curve $\cS(n)$.
\end{theorem}
\begin{proof}
Let $\a_k^n$ be as defined in \eqref{c2-epsilon}.
We restrict our
attention to the subspace $\cH_n$ of $C^{1,\g}_0(\overline B_n)$ given by
\begin{equation}\label{2.3}
\cH_n:=\{v\in C^{1,\g}_0(\overline B_n) \, ,\,
\hbox{s.t. }v(x_1,\dots,x_N)=v(g(x_1,\dots, x_{N-1}),x_N)\, ,\hbox{
  for any }\atop g\in O(N-1)\}
\end{equation}
where $O(N-1)$ is the orthogonal group in  $\R^{N-1}$.\\
By a result of Smoller and Wasserman in
\cite{SW86}, we have that for any $k$ the eigenspace  $V_k$ of the
Laplace-Beltrami operator on  $S^{N-1}$, spanned by the
eigenfunctions corresponding to the
eigenvalue  $\mu _k$ which are $O(N-1)$ invariant, is
one-dimensional.  This implies that, (see Section \ref{s4})
\begin{equation}\label{2.4}
m(\a_k^n+\d)-m(\a_k^n-\d)=1
\end{equation}
if $\d>0$ is
small enough, where $m(\a)$ is the Morse index of the radial solution $u_{n,\a}$ in
the space $\cH_n$. \\
Let us consider the operator $T^n(\a,v):(0,+\infty)\times \cH_n\to \cH_n$,
defined by  $T^n(\a,v):=\left(-\D\right)^{-1}\left( |x|^{\a}|v+\c_{\a}(n)|^{p_{\a}-1}(v+\c_{\a}(n))\right)$, with $\c_{\a}(n)=U_{\a}\left(\frac 1{\e_n}\right)$. $T^n$  is a compact operator
for every fixed $\a$ and is continuous with respect to $\a$. Let us suppose by contradiction that $(\a_k^n,u_{n,\a_k^n})$ is not a
bifurcation point and set
$F^n(\a,v):=v-T^n(\a,v)$. Then there exists $\d_0>0$ such that for
$\d\in (0,\d_0)$ and every $c\in (0,\d_0)$ we have
\begin{equation}\label{2.5}
\begin{array}{ll}
&F^n(\a,v)\neq 0,\quad \forall \,\, \a\in (\a_k^n-\d,\a_k^n+\d),\\
\\
&\forall \,\, v\in \cH_n\hbox{
  such that }\nor v- u_{n,\a} \nor_{\cH_n}\leq c \,\,\hbox{ and }v\neq u_{n,\a}.
\end{array}
\end{equation}
We can also choose $\d_0$
in such a way that the interval $[\a_k^n-\d,\a_k^n+\d]$ does not contain other points $\a_h^n$ with $h\neq k$.
Let us consider the
set  $\Gamma:=\{(\a,v)\in[\a_k^n-\d,\a_k^n+\d]\times  \cH_n\,:\, \nor v- u_{n,\a}\nor_{\cH_n}
<c\}$.  Notice that
$F^n(\a,\cdot)$ is a compact perturbation of the identity and so it makes
sense to consider the Leray-Schauder topological degree $\mathit{deg}
\left(  F^n(\a,\cdot),\Gamma_{\a},0\right)$ of $F^n(\a,\cdot)$ on the set
 $\Gamma_{\a}:=\{v\in \cH_n\hbox{ such that } (\a,v)\in \Gamma\}$. From
(\ref{2.5}) it follows that there  exist no  solutions of $F^n(\a,v)=0$ on $\de
_{[\a_k^n-\d,\a_k^n+\d]\times \cH_n}\Gamma$. By the homotopy invariance of
the degree, we get
\begin{equation}\label{2.6}
\mathit{deg} \left( F^n(\a,\cdot),\Gamma_{\a},0\right)\hbox{ is constant on }[\a_k^n-\d,\a_k^n+\d].
\end{equation}
Since the linearized  operator
$T^n_v(\a,u_{n,\a})$ is invertible for  $\a=\a_k^n+\d$ and
$\a=\a_k^n-\d$, we have
$$\mathit{deg} \left( F^n(\a_k^n-\d,\cdot),\Gamma_{\a_k^n-\d},0\right)=(-1)^{m(\a_k^n-\d)}$$
and
$$\mathit{deg} \left( F^n(\a_k^n+\d,\cdot),\Gamma_{\a_k^n+\d},0\right)=(-1)^{m(\a_k^n+\d)}.$$
By the choice of  $\a_k^n$ and of the space  $\cH_n$ and by (\ref{2.4}) we get
$$\mathit{deg} \left( F^n(\a_k^n-\d,\cdot),\Gamma_{\a_k^n-\d},0\right)=-\mathit{deg} \left(
F^n(\a_k^n+\d,\cdot),\Gamma_{\a_k^n+\d},0\right)$$
contradicting (\ref{2.6}). Then $(\a_k^n,u_{n,\a_k^n})$ is a bifurcation point
and the bifurcating solutions are nonradial since $u_{n,\a} $ is radially
nondegenerate for any $\a$ as proved in Lemma \ref{lemma-radial-nondegeneracy}.
\end{proof}
\noindent Let us observe that these bifurcating solutions lie in the space $\cH_n$
and hence are $O(N-1)$-invariant.\\

\begin{theorem}\label{t-bif-2}
Let $\a_k^n$ be as defined in \eqref{c2-epsilon} and related to some $k$ even.
Then there exist $[\frac N2]$ distinct nonradial solutions of \eqref{p-epsilon} bifurcating from $(\a_k^n,u_{n,\a_k^n})$.
\end{theorem}
\begin{proof}
Let us consider the subgroups $\mathcal{G}_h$ of $O(N)$ defined by
\begin{equation}\label{g-h}
\mathcal{G}_h=O(h)\times O(N-h)\quad \text{ for }1\leq h\leq \left[\frac N2\right].
\end{equation}
In \cite{SW90} (see also  \cite{W89})  it is showed that if $k$ is even then the eigenspace $V_k$ of the
Laplace-Beltrami operator on  $S^{N-1}$, related to the
eigenvalue  $\mu _k$ invariant by the action of $\mathcal{G}_h$, has dimension one. \\
Then defining by $\cH_n^h$ the subspace of $C^{1,\g}_0(\overline B_n)$ of functions invariant by the action of $\mathcal{G}_h$ we have that as in  \eqref{2.4} 
\begin{equation}\label{2.7}
m^h(\a_k^n+\d)-m^h(\a_k^n-\d)=1
\end{equation}
if $\d>0$ is
small enough, where $m^h(\a)$ is the Morse index of the radial solution $u_{n,\a}$ in
the space $\cH_n^h$. \\
Reasoning exactly as in the proof of the previous theorem we get that $(\a_k^n,u_{n,\a_k^n})$ is a bifurcation point and the bifurcating solution is invariant by the action of $\mathcal{G}_h$.\\
Moreover, if we get a solution $v$ which is invariant with respect to the action of the two groups $\mathcal{G}_{h_1}$ and $\mathcal{G}_{h_2}$ with $h_1\neq h_2$, then $v$ must be radial (see \cite{SW90}), and this is not possible, since the radial solutions $u_{n,\a}$ are isolated.\\
Then, we derive the existence of $[\frac N2]$ distinct nonradial solutions of \eqref{p-epsilon} bifurcating from $(\a_k^n,u_{n,\a_k^n})$.
\end{proof}
Let us denote by $\Sigma_n$ the closure in $(0,+\infty)\times \cH_n$ of the set of solutions of $F^n(\a,v)=0$ different from $u_{n,\a}$, i.e
\begin{equation}\label{sigma-n}
\Sigma_n:=\overline{\{(\a,v)\in (0,+\infty)\times \cH_n \, ;\,F^n(\a,v)=0 \, ,\, v\neq u_{n,\a}\}}
\end{equation}
where $F^n(\a,v)$ and $\cH_n$ are as in the proof of Theorem \ref{t-bif} or \ref{t-bif-2}. If $(\a_k^n,u_{n,\a_k^n})\in \cSn$ is a nonradial bifurcation point, then $(\a_k^n,u_{n,\a_k^n})\in \Sigma_n$.\\
For $(\a_k^n,u_{n,\a_k^n})\in \Sigma_n$ we will call $\mathcal{C}(\a_k^n)\subset \Sigma_n$ the closed connected component of $\Sigma_n$ which contains $(\a_k^n,u_{n,\a_k^n})$ and it is maximal with respect to the inclusion.\\
We have the following:
\begin{proposition}\label{v-pos}
Let $\a_k^n$ be as defined in \eqref{c2-epsilon}.
If $(\a,v_\a)\in \mathcal{C}(\a_k^n)$ then $v_\a$ is a solution of \eqref{p-epsilon} corresponding to $\e_n$, in particular $v_\a>0$ in $B_n$.
\end{proposition}
\begin{proof}
Let us consider the subset $\mathcal{C}\subset \mathcal{C}(\a_k^n)$ of points $(\a,v_\a)$ which are non-negative solutions of $ F^n(\a,v)=0$. Obviously $(\a_k^n,u_{n,\a_k^n})\in \mathcal{C}$. We will prove that $\mathcal{C}$ is closed and open in $\mathcal{C}(\a_k^n)$, hence $\mathcal{C}=\mathcal{C}(\a_k^n)$ since $\mathcal{C}(\a_k^n)$ is connected. Moreover, the maximum principle implies that if $(\a,v_{\a})\in \mathcal{C}$ then either $v_{\a}>0$ or $v_{\a}\equiv 0$, but, since zero is not a solution of \eqref{p-epsilon} then the solutions on $\mathcal{C}$ are positive solutions.\\
If $(\a,v_\a)$ is a point in the closure of $\mathcal{C}$ then there is a sequence of points $(\a_j,v_j)$ in $\mathcal{C}$ that converges to $(\a,v_\a)$ in $(0,+\infty)\times C^{1,\g}_0(\overline B_n)$.
As $j\to +\infty$ we get that $v_\a$ is a solution of $F^n(\a,v_\a)=0$ and $v_\a\geq 0$ in $B_n$. By the Maximum principle either $v_\a>0$ or $v_\a\equiv 0$ in $B_n$. But the second case is not possible since zero is not a solution of \eqref{p-epsilon}. Then $v_\a>0$ in $B_n$, $(\a,v_\a)\in\mathcal{C}$ and $\mathcal{C}$ is closed.

Now we will show that $\mathcal{C}$ is open in $ \mathcal{C}(\a_k^n)$. Let $ (\a,v_\a)$ be a point in $\mathcal{C}$ and
$(\bar{\a},v_{\bar{\a}})$ in $ \mathcal{C}(\a_k^n)$ such that $\nor v_{\a}-v_{\bar \a}\nor_{\mathcal{H}_n}< \c_{\bar{\a}}(n)$, then
\begin{align*}
 - \Delta v_{\bar{\a}} &= |x|^{\bar{\a}}|v_{\bar{\a}}+\c_{\bar{\a}}(n)|^{p_{\bar{\a}}-1}(v_{\bar{\a}} +\c_{\bar{\a}}(n)) \\
 &= |x|^{\bar{\a}}|v_{\bar{\a}}+\c_{\bar{\a}}(n)|^{p_{\bar{\a}}-1}(v_\a + v_{\bar{\a}} - v_\a +\c_{\bar{\a}}(n)) \geq 0
\end{align*}
in $B_n$ and, since $v_{\bar{\a}} = 0$ on $\partial B_n$, it follows by the maximum principle that $v_{\bar{\a}} >0$ in $B_n$.

\end{proof}
\begin{theorem}\label{t-alternative}
Let $\a_k^n$ be as defined in \eqref{c2-epsilon} and let $\mathcal{C}(\a_k^n)$ be the closed connected component of $\Sigma_n$ which contains $(\a_k^n,u_{n,\a_k^n})$ and it is maximal with respect to the inclusion.
Then either
\begin{itemize}
\item[i)] $\mathcal{C}(\a_k^n)$ is unbounded in $[0,+\infty)\times \cH_n$, or
\item[ii)] there exists $\a_h^n$ with $h\neq k$ such that $(\a_h^n,u_{n,\a_h^n})\in\mathcal{C}(\a_k^n)$, or
\item[iii)] $\mathcal{C}(\a_k^n)$ meets $\{0\}\times \cH_n$.
\end{itemize}
\end{theorem}
\begin{proof}
The proof follows from the global bifurcation result of Rabinowitz, \cite{R71}. One can see also \cite{AM07} or \cite{G10} for details.
\end{proof}
\begin{remark}\label{rem-3.10}
\rm The results of Proposition \ref{v-pos} and  Theorem \ref{t-alternative} hold for every bifurcation point  generated by an odd change in the Morse index of the radial solution $u_{n,\a}$. Then, using Theorem \ref{t-bif-2}, when $k$ is even we can find $\big[\frac N2\big]$ different continua of (positive) nonradial solutions bifurcating from $(\a_k^n,u_{n,\a_k^n})$. Moreover these continua are global in the sense that they satisfy Theorem \ref{t-alternative}.
\end{remark}
\sezione{Some estimates on the approximating solutions}\label{s6}
In this section we give some estimates on the decay of solutions (not necessary radial) of \eqref{p-epsilon} as $\e\to 0$. As before we consider the functions defined in $\R^N$ extended by zero outside of $B_{\frac 1{\e}}(0)$, and we denote by $\nor\cdot\nor_{\b}$ and $\nor\cdot\nor_{1,2}$ the norm of $L^{\infty}_{\b}(\R^N)$ and $D^{1,2}(\R^N)$ respectively. Let $X$ be as defined in \eqref{X}. Then we have the following
\begin{proposition}\label{prima-stima}
Let $\e_n$ and $\a_n$ be sequences such that $\e_n\to 0$ and $\a_n\to \bar \a>0$ as $n\to +\infty$. Let $v_{n}$ be a sequence of solutions of \eqref{p-epsilon} in $B_n:=B_{\frac 1{\e_n}}(0)$, corresponding to the exponent $\a=\a_n$, i.e.
\begin{equation}\label{generale}
\left\{
\begin{array}{ll}
-\Delta v_{n}=C(\a_n)|x|^{\a_n}\left(v_{n}+\c(n)\right)^{p_{\a_n}}&\text{ in } B_n, \\
v_{n}>0 &\text{ in } B_n,\\
v_{n}=0,&\text{ on } \partial B_n ,
\end{array}
\right.
\end{equation}
where $\c(n):=U_{\a_n}\left(\frac 1 {\e_n}\right)=\frac{\e_n^{N-2}}{\left( 1+\e_n^{2+\a_n}\right)^{\frac{N-2}{2+\a_n}}}$.
Assume that $\nor v_{n}\nor_X\leq A$ for some positive constant $A$ and for every $n$. Then there exists $C>0$ such that
$$v_{n}(x)\leq \frac C{\left(1+|x|\right)^{N-2}} \text{ for every $x\in \R^N$ and for every $n\in \N$}.$$
\end{proposition}
\begin{proof}
We first note that
\begin{equation}\label{6.2}
\c(n) \leq \frac{C}{(1+|x|)^{N-2}} \qquad \forall \,\,\, x \in B_{n}
\end{equation}
and, since $\nor v_{n}\nor_{\b}\le\nor v_{n}\nor_X\leq A$ with $\b < N-2$, we have
\begin{equation}\label{6.3}
|v_{n}(x)+\c(n)| \leq \frac{C}{(1+|x|)^\b} \qquad \forall \,\,\, x \in B_{n}.
\end{equation}
We shall use the integral representation of $v_n$ to obtain the desired estimate. If $G_n(x,y)$ denotes the Green function of $B_n$ then
\begin{align}
v_n(x) & = \int_{B_n}G_n(x,y)|y|^{\a_n}(v_n(y)+\c(n))^{p_{\a_n}}dy \notag \\
       & \leq C \int_{B_n}G_n(x,y)\frac{|y|^{\a_n}}{(1+|y|)^{\b p_{\a_n}}}dy .\label{6.4}
\end{align}
Now we consider the function
$$ \psi_n(x) = \int_{B_n}G_n(x,y)\frac{|y|^{\a_n}}{(1+|y|)^{\b p_{\a_n}}}dy$$
which verifies
\begin{equation*}
\begin{cases}
-\Delta \psi_n = \frac{|x|^{\a_n}}{(1+|x|)^{\b p_{\a_n}}}, & \text{in} \,\,\, B_n \\
\psi_n = 0 , & \text{on} \,\,\, \partial B_n.
\end{cases}
\end{equation*}
So $\psi_n(r) = \psi_n(|x|)$ satisfies
\begin{equation}\label{6.5}
 -r^{N-1}\psi_n'(r) = \int_{0}^{r} \frac{s^{\a_n + N-1}}{(1+s)^{\b p_{\a_n}}} ds.
\end{equation}
Since $\b > \frac{N}{N+2}(N-2) > \frac{N+\a_n}{p_{\a_n}}$ for all $n$, we get that
$$ \frac{s^{\a_n + N-1}}{(1+s)^{\b p_{\a_n}}} \in L^1(\R_+)$$
and consequently, by \eqref{6.5}, we have $-\psi_n'(r) \leq C r^{-(N-1)}$. Therefore
\begin{equation} \label{6.6}
\psi_n(r)
 \leq \frac{C}{r^{N-2}}
\end{equation}
and, since  $\psi_n$ is bounded then $\psi_n(r) \leq \frac{C}{(1+r)^{N-2}}$.
The claim follows by \eqref{6.4} and \eqref{6.6}.
\end{proof}
\begin{proposition}\label{lambda=1}
Let $\e_n$ and $\a_n$ be sequences such that $\e_n\to 0$ and $\a_n\to \bar \a>0$ as $n\to +\infty$. Let $v_{n}$ be a sequence of solutions of \eqref{generale}
in $B_n$ related to the exponents $\a_n$. 
If $v_{n}\to U_{\l,\bar \a}$ in $X$ then we have that $\l=1$.
\end{proposition}
\begin{proof}
By the Pohozaev identity we get
\begin{equation}\label{6.7}
\begin{split}
&-\frac{N\o_N\c(n)^{p_{\a_n}+1}}{(p_{\a_n}+1) \e_n^{\a_n+N}}  + \frac{N-2}{2}\c(n)C(\a_n)\int_{B_{n}}|x|^{\a_n}(v_n(x)+\c(n))^{p_{\a_n}}dx
\\&= \frac{1}{2\e_n}\int_{\partial B_n}\left( \frac{\partial v_n}{\partial \nu}\right)^2 d\sigma_x.
\end{split}
\end{equation}
Since $v_{n}\to U_{\l,\bar \a}$ in $X$, by the standard regularity theory, it follows that
\begin{equation} \label{6.8}
v_{n}\to U_{\l,\bar \a}  \quad \text{in} \,\,\, C_{loc}^2(\R^N) .
\end{equation}
Recalling that $\c(n) = \e_n^{N-2} (1+o_n(1)) $ we have
\begin{equation*}
\frac{\c(n)^{p_{\a_n}+1}}{\e_n^{\a_n+N}} = O(\e_n^{(N-2)(p_{\a_n}+1)-(\a_n+N)}) = O(\e_n^{N+\a_n})
\end{equation*}
and by Proposition \ref{prima-stima} and \eqref{6.8} we derive
\begin{equation*}
\int_{B_{n}}|x|^{\a_n}(v_n(x)+\c(n))^{p_{\a_n}}dx = \int_{\R^N}|x|^{\bar{\a}}U_{\l,\bar \a}^{p_{\bar{\a}}}(x)dx + o_n(1)
\end{equation*}
Thus \eqref{6.7} becomes
\begin{align}
&\e_n^{N-1}\left( (N-2) C(\bar{\a}) \int_{\R^N}|x|^{\bar{\a}}U_{\l,\bar \a}^{p_{\bar{\a}}}(x)dx + o_n(1)\right)
= \int_{\partial B_n}\left( \frac{\partial v_n}{\partial \nu}\right)^2 d\sigma_x. \label{6.9}
\end{align}
Next we expand the right-hand side of \eqref{6.9} using the Green function $G_n$ of $B_n$. For $x \in \partial B_n$, we write
\begin{equation}\label{6.10}
\frac{\partial v_n}{\partial \nu}(x) = C(\a_n)\int_{B_n}\frac{\partial G_n}{\partial \nu_x}(x,y)|y|^{\a_n}(v_n(y)+\c(n))^{p_{\a_n}}dy
\end{equation}
and substituting (here $\o_N$ denotes the volume of the
$N-$dimensional unit ball)
$$ \frac{\partial G_n}{\partial \nu_x}(x,y) = -\frac{1}{N\o_N}\frac{1-\e_n^2|y|^2}{\e_n|x-y|^N}$$
into \eqref{6.10} gives
\begin{align}
\frac{\partial v_n}{\partial \nu}(x) &= -(C(\bar{\a}) + o_n(1))\int_{B_n}\frac{1}{N\o_N}\frac{1-\e_n^2|y|^2}{\e_n|x-y|^N}|y|^{\a_n}(v_n(y)+\c(n))^{p_{\a_n}}dy
\notag \\
& = -\frac{\e_n^{N-1}(C(\bar{\a}) + o_n(1))}{N\o_N}\int_{B_n}\frac{(1-\e_n^2|y|^2)|y|^{\a_n}}{|z-\e_n y|^N}(v_n(y)+\c(n))^{p_{\a_n}}dy \label{6.11}
\end{align}
where we have posed $x=\frac{z}{\e_n}, \,\,\, z \in S^{N-1}$. Let us assume that we have proved
\begin{equation}\label{6.12}
\int_{B_n}\frac{(1-\e_n^2|y|^2)|y|^{\a_n}}{|z-\e_n y|^N}(v_n(y)+\c(n))^{p_{\a_n}}dy = \int_{\R^N}|y|^{\bar{\a}}U_{\l,\bar \a}^{p_{\bar{\a}}}(y)dy + o_n(1)
\end{equation}
uniformly with respect to $z \in S^{N-1}$, then \eqref{6.11} becomes
\begin{equation}\label{6.13}
\frac{\partial v_n}{\partial \nu}(x) = -\frac{\e_n^{N-1}C(\bar{\a})}{N\o_N}\left(\int_{\R^N}|y|^{\bar{\a}}U_{\l,\bar \a}^{p_{\bar{\a}}}(y)dy + o_n(1) \right)
\end{equation}
where the term $o_n(1)$ is uniform in $x \in \partial B_n$. Thus
$$ \left(\frac{\partial v_n}{\partial \nu}(x)\right)^2 = \frac{\e_n^{2(N-1)}\left(C(\bar{\a})\right)^2}{N^2\o_N^2}\left[ \left(\int_{\R^N}|y|^{\bar{\a}}U_{\l,\bar \a}^{p_{\bar{\a}}}(y)dy\right)^2 + o_n(1) \right] $$
and
\begin{align*}
\int_{\partial B_n} \left(\frac{\partial v_n}{\partial \nu}(x)\right)^2 d\sigma_x
&= \frac{\e_n^{(N-1)}\left(C(\bar{\a})\right)^2}{N\o_N}\left[ \left(\int_{\R^N}|y|^{\bar{\a}}U_{\l,\bar \a}^{p_{\bar{\a}}}(y)dy\right)^2 + o_n(1) \right].
\end{align*}
Turning back to \eqref{6.9} we have
\begin{equation*}
 (N-2) \int_{\R^N}|x|^{\bar{\a}}U_{\l,\bar \a}^{p_{\bar{\a}}}(x)dx + o_n(1)
= \frac{C(\bar{\a})}{N\o_N} \left(\int_{\R^N}|y|^{\bar{\a}}U_{\l,\bar \a}^{p_{\bar{\a}}}(y)dy\right)^2
\end{equation*}
and passing to the limit
\begin{align}
(N-2) \int_{\R^N}|x|^{\bar{\a}}U_{\l,\bar \a}^{p_{\bar{\a}}}(x)dx
&= \frac{C(\bar{\a})}{N\o_N} \left(\int_{\R^N}|y|^{\bar{\a}}U_{\l,\bar \a}^{p_{\bar{\a}}}(y)dy\right)^2 \notag \qquad \Longrightarrow\\
(N-2) &= \frac{C(\bar{\a})}{N\o_N} \int_{\R^N}|y|^{\bar{\a}}U_{\l,\bar \a}^{p_{\bar{\a}}}(y)dy. \label{6.14}
\end{align}
A straightforward computation gives
\begin{align}
\int_{\R^N}|y|^{\bar{\a}}U_{\l,\bar \a}^{p_{\bar{\a}}}(y)dy =  \frac{N\o_N}{\lambda^{\frac{N-2}{2}}} \frac{1}{N+\bar{\a}} \label{6.15}
\end{align}
Then, by \eqref{6.14} and \eqref{6.15}, we infer
\begin{equation}
N-2 = \frac{C(\bar{\a})}{\lambda^{\frac{N-2}{2}}} \frac{1}{N+\bar{\a}} = \frac{N-2}{\lambda^{\frac{N-2}{2}}} \quad \Longrightarrow \quad \lambda = 1.
\label{6.16}
\end{equation}
which gives the claim.\\
It remains to verify \eqref{6.12}, which is a straightforward calculation. Using the decay of $v_n$ and the Lebesgue's dominated convergence theorem
we get
\begin{equation*}
\int_{|y|\leq \frac{1}{2\e_n}}\frac{(1-\e_n^2|y|^2)|y|^{\a_n}}{|z-\e_n y|^N}(v_n(y)+\c(n))^{p_{\a_n}}dy \quad \longrightarrow \quad
\int_{\R^N}|y|^{\bar{\a}}U_{\l,\bar \a}^{p_{\bar{\a}}}(y)dy
\end{equation*}
uniformly with respect to $z$ as $\e_n \to 0$, recalling that $|z| =1$. Finally we estimate the integral in the rest of the domain $ \frac{1}{2\e_n} \leq |y| \leq \frac{1}{\e_n}$.
Using again the decay of $v_n$ we have
\begin{align*}
\int_{\frac{1}{2\e_n} \leq |y| \leq \frac{1}{\e_n}}&\frac{(1-\e_n^2|y|^2)|y|^{\a_n}}{|z-\e_n y|^N}(v_n(y)+\c(n))^{p_{\a_n}}dy  \\
&\leq C \e_n^{N+2+\a_n} \int_{\frac{1}{2\e_n} \leq |y| \leq \frac{1}{\e_n}}\frac{1-\e_n^2|y|^2}{|z-\e_n y|^N}dy \\
& \leq C \e_n^{2 + \a_n} \int_{\frac{1}{2} \leq |\xi-z| \leq 1}\frac{1-|z-\xi|^2}{|\xi|^N}d\xi \\
& \leq C \e_n^{2 + \a_n} \int_{\frac{1}{2} \leq |\xi-z| \leq 1}\frac{2 + |\xi|}{|\xi|^{N-1}}d\xi \quad \longrightarrow \quad 0, \qquad \text{as} \quad \e_n \to 0,
\end{align*}
uniformly with respect to $z$. The proof is now complete.
\end{proof}
\begin{lemma}\label{decadimento-differenza}
Let $\e_n$ and $\a_n$ be sequences such that $\e_n\to 0$ and $\a_n\to \bar \a>0$ as $n\to +\infty$. Let $v_{n}$ be a sequence of nonradial solutions of \eqref{generale}
in $B_n$ related to the exponents $\a_n$. If $v_n \to U_{\bar \a}$ in $X$ then there exists a constant $C>0$ such that
\begin{equation}\label{decadimento-diff}
\frac{|u_n-v_n|}{\,\,\,\nor u_n-v_n\nor_{\infty}}\leq \frac C{\,\,\left( 1+|x|\right)^{N-2}}
\end{equation}
for any $n$ sufficiently large, where $u_n$ is the radial solution of \eqref{generale} (corresponding to $\a=\a_n$) as defined in \eqref{w-epsilon}.
\end{lemma}
\begin{proof}
We let $w_n:=\frac{u_n-v_n}{\,\,\nor u_n-v_n\nor_{\infty}}$. Then $w_n$ satisfies
\begin{equation}\label{differenza}
\left\{\begin{array}{ll}
-\Delta w_n =C(\a_n)|x|^{\a_n}a_n(x)w_n & \text{ in }B_n\\
w_n=0 & \text{ on } \de B_n
\end{array}
\right.
\end{equation}
where $a_n(x)=p_n \int_0^1\left( tu_n+(1-t)v_n+\c(n)\right)^{p_n-1}\, dt$, for  $p_n:=p_{\a_n}$ and $\c(n)$ as before. By hypothesis $v_n\to U_{\bar \a}$ in $X$ and,  by Proposition \ref{prima-stima} and Proposition \ref{p3.2}, there exists a constant $C>0$ (independent on $n$), such that
$$|v_n(x)|\leq \frac C{\left(1+|x|\right)^{N-2}}\quad \text{ and }\quad |u_n(x)|\leq \frac C{\left(1+|x|\right)^{N-2}}.$$
Then, by \eqref{6.2}, we have
\begin{align}
|a_n(x)|&\leq C\left(|u_n|^{p_n-1}+|v_n|^{p_n-1}+\left(\c(n)\right)^{p_n-1}\right) \notag \\
&\leq C\left(\frac 1{\left(1+|x|\right)^{(N-2)(p_n-1)}}\right) \notag \\
&\leq \frac C{\left(1+|x|\right)^{4+2\a_n}}.
\label{a_n}
\end{align}
We consider the Kelvin transform of $w_n$, i.e.
$$\widehat w_n(x):=\frac 1{|x|^{N-2}}w_n\left( \frac x{\,\,|x|^2}\right),\quad x\in \R^N\setminus B_{\e_n}.$$
It satisfies
$$
\left\{\begin{array}{ll}
-\Delta \widehat w_n =C(\a_n)\frac 1{|x|^{4+\a_n}}a_n(\frac x{\,\,|x|^2})\widehat w_n & \text{ in }R^N\setminus B_{\e_n} \\
\widehat w_n=0 & \text{ on } \de B_{\e_n}
\end{array}
\right.
$$
and, using \eqref{a_n},
$$\frac 1{\,\,|x|^{4+\a_n}}a_n\left(\frac x{\,\,|x|^2}\right)\leq \frac 1{|x|^{4+\a_n}}\frac {C}{\left( 1+\frac 1{|x|}\right)^{4+2\a_n}}=\frac {|x|^{\a_n}}{\left(1+|x|\right)^{4+2\a_n}}.$$
Then, since $\widehat w_n=0$ on $\de B_{\e_n}$, the regularity theory (see \cite{H}) implies
\begin{equation}\label{han}
\nor \widehat w_n\nor_{L^{\infty}(B_1 \backslash B_{\e_n})}\leq C\nor \widehat w_n\nor_{L^{2^*}(B_2 \backslash B_{\e_n} )}.
\end{equation}
We will show that there exists a constant $C$ (independent on $n$) such that $\nor \widehat w_n\nor_{L^{2^*}(B_2 \backslash B_{\e_n})}\leq C$ and then  \eqref{han} will imply \eqref{decadimento-diff}.\\
Using the Sobolev embedding Theorem, \eqref{differenza} and \eqref{a_n} we have
\begin{align*}
  &\nor  w_n\nor_{L^{2^*}(B_n)}^2\leq \frac 1{S} \int_{B_n}|\na w_n|^2\, dx=\frac{C(\a_n)}{S}\int_{B_n} |x|^{\a_n}a_n(x)w_n^2\, dx \\
&
\leq C\int_{B_n}\frac 1{\,\,\left(1+|x|\right)^{4+\a_n}}|w_n|^2\, dx\leq C\int_{B_n}\frac 1{\,\,\left(1+|x|\right)^{4+\a_n}}|w_n|^{2-\d}\, dx
\end{align*}
since $|w_n|\leq 1$, for some $\d>0$ that we will choose later. Now, using the H\"older inequality, we get
\begin{align*}
\nor  w_n\nor_{L^{2^*}(B_n)}^2&\leq\left(\int_{B_n}\left(\frac 1{\,\,\left(1+|x|\right)^{4+\a_n}}\right)^{\frac{2N}{4+\d(N-2)}}\, dx\right)^ {\frac{4+\d(N-2)}{2N}} \left( \int_{B_n}|w_n|^{2^*}\, dx\right)^{\frac {2-\d}{2^*}}\\
&\leq C_{\d}\left( \int_{B_n}|w_n|^{2^*}\, dx\right)^{\frac {2-\d}{2^*}}= C_{\d}\nor  w_n\nor_{L^{2^*}(B_n)}^{2-\d}
\end{align*}
if $0<\d<\min\{2, \frac{4+\bar \a}{N-2} \}$ where $\bar \a=\lim \a_n$. Note that the constant $C_{\d}$ is independent on $n$, for $n$ large enough, and
using \eqref{han} we get
$$\nor \widehat w_n\nor_{L^{\infty}(B_1 \backslash B_{\e_n})}\leq C\nor \widehat w_n\nor_{L^{2^*}(\R^N \backslash B_{\e_n})} =
C\nor w_n\nor_{L^{2^*}(B_n)} \leq C.$$
This concludes the proof.
\end{proof}
\begin{proposition}\label{pohozaev}
Let $\e_n$ and $\a_n$ be a sequences such that $\e_n\to 0$ and $\a_n\to \bar \a>0$ as $n\to +\infty$. Let $v_{n}$ be a sequence of nonradial solutions of \eqref{generale}
in $B_n$ related to the exponents $\a_n$. If $\bar \a\neq \a_k$ for all $k\in \N$ then there exists a constant $c>0$ (independent on $n$), such that
\begin{equation}
\nor v_n-u_n\nor_{\infty}\geq c
\end{equation}
where $u_n$ is the radial solution of \eqref{generale} in $B_n$ (corresponding to the exponents $\a_n$) as defined in \eqref{w-epsilon}.
\end{proposition}
\begin{proof}
Let us suppose, by contradiction, that there exists a sequence of nonradial solutions $v_n$ of \eqref{generale} in $B_n$ related to the exponent $\a_n$
such that
\begin{equation} \label{6.22}
\nor v_n-u_n\nor_{\infty}\to 0 \qquad \text{as}\,\,\, n\to +\infty.
\end{equation}
Both solutions satisfy the Pohozaev identity \eqref{6.7}, so if we write the identities for $u_n$ and $v_n$ and subtract one from another we obtain
\begin{equation}\label{6.23}
\begin{split}
(N-2)\e_n \c(n)C(\a_n)\int_{B_n}|x|^{\a_n}\left[ (u_n(x) + \c(n))^{p_{\a_n}} - (v_n(x) + \c(n))^{p_{\a_n}}\right]dx \\=
\int_{\partial B_n}\frac{\partial}{\partial \nu}(u_n(x)-v_n(x))\frac{\partial}{\partial \nu}(u_n(x)+v_n(x))d\sigma_x.
\end{split}
\end{equation}
By mean value theorem we get
\begin{align*}
\begin{split}
&(u_n(x) + \c(n))^{p_{\a_n}} - (v_n(x) + \c(n))^{p_{\a_n}} \\&\qquad =p_{\a_n}\int_{0}^{1}\left( tu_n(x) + (1-t)v_n(x)+\c(n)\right)^{p_{\a_n}-1}dt\,\,(u_n(x) -v_n(x)) \\
&\qquad =a_n(x)(u_n(x) -v_n(x))
\end{split}
\end{align*}
and setting $w_n = \frac{u_n - v_n}{\nor v_n-u_n\nor_{\infty}}$  \eqref{6.23} becomes
\begin{equation}\label{6.24}
(N-2)\e_n \c(n)C(\a_n)\int_{B_n}|x|^{\a_n}a_n(x)w_n(x)dx =
\int_{\partial B_n}\frac{\partial w_n}{\partial \nu}\frac{\partial}{\partial \nu}(u_n(x)+v_n(x))d\sigma_x.
\end{equation}
Note that, by \eqref{convergence} and \eqref{6.22}, we have that
$$ a_n \to p_{\bar{\a}}U_{\bar{\a}}^{p_{\bar{\a}}-1} \qquad \text{in} \quad C_{loc}^2(\R^N),$$
and the function $w_n$ satisfies the following equation
\begin{equation}\label{6.25}
\begin{cases}
-\Delta w_n = C(\a_n)|x|^{\a_n}a_n(x)w_n, & \text{in}\,\,\,\, B_n \\
w_n = 0, & \text{in}\,\,\,\, \partial B_n.
\end{cases}
\end{equation}
Since $\nor w_n \nor_{\infty} = 1$, by standard regularity theorems, one can prove that $w_n \to w$ in $C_{loc}^2(\R^N)$ where $w$ satisfies
\begin{equation}\label{6.26}
\begin{cases}
-\Delta w = C(\bar{\a})p_{\bar{\a}}|x|^{\bar{\a}}U_{\bar{\a}}^{p_{\bar{\a}}-1}w, & \text{in}\,\,\,\, \R^N \\
|w| \leq 1, & \text{in}\,\,\,\, \R^N.
\end{cases}
\end{equation}
and, because $\bar{\a} \neq \a_k$, by Theorem \ref{linearized} we have
\begin{equation}\label{6.27}
w(x) = A \frac{1-|x|^{2+\bar{\a}}}{(1+|x|^{2+\bar{\a}})^{\frac{N+\bar{\a}}{2+\bar{\a}}}} \qquad \text{for some} \quad A \in \R.
\end{equation}

The next step is to prove that $A=0$. In order to do this, we shall expand both sides of \eqref{6.24} in powers of $\e_n$ as we did in the proof of Proposition \ref{lambda=1} and then pass to the limit.

Let us expand first the LHS of \eqref{6.24}. Using the decay properties of $w_n$, $u_n$, $v_n$ and Lebesgue's theorem we have
$$ \int_{B_n}|x|^{\a_n}a_n(x)w_n(x)dx = p_{\bar{\a}}\int_{\R^N}|x|^{\bar{\a}}w(x)U_{\bar{\a}}^{p_{\bar{\a}}-1}(x)dx + o_n(1).$$
Hence the LHS of \eqref{6.24} becomes
\begin{align}\label{6.28}
LHS = (N-2)C(\bar{\a})\e_n^{N-1} \left( p_{\bar{\a}}\int_{\R^N}|x|^{\bar{\a}}w(x)U_{\bar{\a}}^{p_{\bar{\a}}-1}(x)dx + o_n(1) \right)
\end{align}

Now, we expand the RHS of \eqref{6.24}. For $x \in \partial B_n$, as in \eqref{6.10}-\eqref{6.13}, we can write
\begin{align}
\frac{\partial w_n}{\partial \nu}(x) &= C(\a_n)\int_{B_n}\frac{\partial G_n}{\partial \nu_x}(x,y)|y|^{\a_n}a_n(y)w_n(y)dy \notag \\
&\left(\text{setting}\,\,\, z=\e_n x, \,\,\, z \in S^{N-1}\right) \notag \\
&= -\frac{\e_n^{N-1}C(\a_n)}{N\o_N}\int_{B_n}\frac{(1-\e_n^2|y|^2)|y|^{\a_n}}{|z-\e_n y|^N}a_n(y)w_n(y)dy   \notag \\
&= -\frac{\e_n^{N-1}C({\bar \a})}{N\o_N}\left( \int_{\R^N}|y|^{\bar \a}p_{\bar \a}U_{\bar \a}^{p_{\bar \a}-1}(y)w(y)dy + o_n(1) \right). \label{6.29}
\end{align}
Also, by \eqref{6.13}, one has
\begin{align} \label{6.30}
\frac{\partial u_n}{\partial \nu}(x)+ \frac{\partial v_n}{\partial \nu}(x)= -2\frac{\e_n^{N-1}C(\bar{\a})}{N\o_N}\left(\int_{\R^N}|y|^{\bar{\a}}U_{\l,\bar \a}^{p_{\bar{\a}}}(y)dy + o_n(1) \right)
\end{align}
Hence, by \eqref{6.29} and \eqref{6.30} we obtain
\begin{align}
&\int_{\partial B_n}\frac{\partial w_n}{\partial \nu}\frac{\partial}{\partial \nu}(u_n(x)+v_n(x))d\sigma_x& \notag \\
\begin{split}
&=2\frac{\e_n^{2(N-1)}\left(C(\bar{\a})\right)^2}{N^2\o_N^2}
\Bigg[ \left( \int_{\R^N}|y|^{\bar \a}p_{\bar \a}U_{\bar \a}^{p_{\bar \a}-1}(y)w(y)dy\right)
 \left( \int_{\R^N}|y|^{\bar{\a}}U_{\l,\bar \a}^{p_{\bar{\a}}}(y)dy \right)  \\
 & \qquad \qquad \qquad \qquad \qquad  + o_n(1) \Bigg] \left| \partial B_n\right|
\end{split} \notag \\
\begin{split}
&=2\frac{\e_n^{(N-1)}\left(C(\bar{\a})\right)^2}{N\o_N}
\Bigg[ \left( \int_{\R^N}|y|^{\bar \a}p_{\bar \a}U_{\bar \a}^{p_{\bar \a}-1}(y)w(y)dy\right)
 \left( \int_{\R^N}|y|^{\bar{\a}}U_{\l,\bar \a}^{p_{\bar{\a}}}(y)dy \right)  \\
 & \qquad \qquad \qquad \qquad \qquad  + o_n(1) \Bigg] . \label{6.31}
\end{split}
\end{align}
Therefore, substituting \eqref{6.28}, \eqref{6.29} and \eqref{6.30} into \eqref{6.24}, canceling the terms which appear on both sides and passing to the limit we get
\begin{equation*}
\begin{split}
&(N-2)\int_{\R^N}|x|^{\bar{\a}}w(x)U_{\bar{\a}}^{p_{\bar{\a}}-1}(x)dx \\&=
2\frac{C(\bar{\a})}{N\o_N}
\left( \int_{\R^N}|y|^{\bar \a}U_{\bar \a}^{p_{\bar \a}-1}(y)w(y)dy\right)
 \left( \int_{\R^N}|y|^{\bar{\a}}U_{\l,\bar \a}^{p_{\bar{\a}}}(y)dy \right)
\end{split}
\end{equation*}
and using \eqref{6.27} we get,
\begin{equation}\label{6.32}
\begin{split}
&A(N-2)\int_{\R^N}|x|^{\bar{\a}}\frac{1-|x|^{2+\bar{\a}}}{(1+|x|^{2+\bar{\a}})^{\frac{N+\bar{\a}}{2+\bar{\a}}}}U_{\bar{\a}}^{p_{\bar{\a}}-1}(x)dx \\&=
2\frac{AC(\bar{\a})}{N\o_N}
\left( \int_{\R^N}|y|^{\bar \a}\frac{1-|y|^{2+\bar{\a}}}{(1+|y|^{2+\bar{\a}})^{\frac{N+\bar{\a}}{2+\bar{\a}}}}U_{\bar \a}^{p_{\bar \a}-1}(y)dy\right)
 \left( \int_{\R^N}|y|^{\bar{\a}}U_{\l,\bar \a}^{p_{\bar{\a}}}(y)dy \right).
\end{split}
\end{equation}
One can verify that
\begin{equation}\label{6.33}
\int_{\R^N}|y|^{\bar \a}\frac{1-|y|^{2+\bar{\a}}}{(1+|y|^{2+\bar{\a}})^{\frac{N+\bar{\a}}{2+\bar{\a}}}}U_{\bar \a}^{p_{\bar \a}-1}(y)dy=
-\frac{N\o_N(N-2)}{C({\bar \a})p_{\bar \a}}  \neq 0
\end{equation}
and by \eqref{6.15} we deduce
\begin{equation}\label{6.34}
A(N-2) = 2\frac{AC(\bar{\a})}{N\o_N} \frac{N\o_N}{N+{\bar \a}} = 2 A (N-2) \quad \Longrightarrow \quad A=0.
\end{equation}

Therefore
\begin{equation} \label{6.35}
w_n \to 0 \quad \text{in} \,\,\, C_{loc}^2(\R^N).
\end{equation}
Let $x_n \in B_n$ be such that $|w_n(x_n)| = 1 = \nor w_n \nor_{\infty}$. By \eqref{decadimento-diff} the sequence $x_n$ remains bounded, but this contradicts \eqref{6.35}.\\
So \eqref{6.22} cannot occur and this gives the claim.
\end{proof}
\sezione{The bifurcation result} \label{s7}
Let us consider the radial solution $U_\a$ of problem \eqref{1} for $\a\in (0,+\infty)$. As shown in Section \ref{s2} the solutions $U_\a$ are always degenerate in the space of radial functions.
Indeed the linearized equation \eqref{1.4}
has the radial solution $Z(x)$, as in \eqref{i12} for any $\a\in (0,+\infty)$.
On the other hand, when $\a$ is even, the kernel of the linearized operator is richer and it is generated by the functions in \eqref{i13}.\\
Moreover, as shown in Corollary \ref{cor-1} the Morse index of $U_\a$ changes as $\a$ crosses $\a_k$, with $\a_k=2(k-1)$ and all the eigenfunctions associated to the linearized
problem, (i.e. the solutions of \eqref{1.12}) lie in the space $X$, defined in \eqref{X}.
Let us define the space
$$\cH:=\{v\in X
\,\text{ s.t. }v(x_1,\dots,x_N)=v(g(x_1,\dots,x_{N-1}),x_N),\,\, \forall \,g\in O(N-1)\}.$$
Using a result of \cite{SW86} we have that the Morse index of $U_\a$ in $\a_k$ increases by one if we restrict to the space $\cH$. Then we get
$$m(\a_k+\d)-m(\a_k-\d)=1$$
if $m$ is the Morse index of $U_\a$ in the space $\cH$. We want to use this change in the Morse index of $U_\a$ (in the space $\cH$) to prove the existence of continua of nonradial solutions of \eqref{1} bifurcating from $(\a_k,U_{\a_k})$ in the product space $(0,+\infty)\times \cH$.\\
But, due to the degeneracy of the radial solution $U_\a$ for any $\a$, we cannot obtain the bifurcation result directly. Then to get the desired result we consider the approximating problem \eqref{p-epsilon}.

\subsection{Proof of the main theorem}
Given a sequence $\e_n\to 0$, we have a sequence of nondegenerate radial solutions $u_{n,\a}$ of \eqref{p-epsilon} (corresponding to $\e=\e_n$), that converges to $U_\a$ as $n\to +\infty$. \\
In Section \ref{s4}, \ref{s5}
we proved that, if $\a_k^n$ satisfy \eqref{c2-epsilon}, then $(\a_k^n,u_{n,\a_k^n})$ are nonradial bifurcation points and give rise to the continua $\mathcal{C}( \a_k^n)$ in the space $(0,+\infty)\times \cH_n$ (but also in the space $(0,+\infty)\times \cH_n^h$, if $k$ is even, where $\cH_n^h$ is as in the proof of Theorem \ref{t-bif-2}). These continua $\mathcal{C}(\a_k^n)$ are global and obey the so called Rabinowitz alternative Theorem, (see Theorem \ref{t-alternative}).\\
In this section we want to prove that these continua $\mathcal{C}(\a_k^n)$ converge in a suitable sense, as $n\to +\infty$ to continua of nonradial solutions of \eqref{1} that bifurcate from $(\a_k,U_{\a_k})$. To do this we use some ideas already used in \cite{AG}, see also \cite{GP11}. \\
Extending the functions by zero outside of $B_n$ and by regularity theorems we can infer that $\mathcal{C}( \a_k^n)$
belongs to the space
$$Z:=(0,+\infty)\times \cH,$$
where $\cH$ is as defined before. Moreover, by Proposition \ref{p3.2} $u_{n,\a_k^n}\to U_{\a_k}$ in $\cH\subset X$ as $n\to +\infty$. \\
To prove the bifurcation result we need the following topological lemma (see Lemma 3.1 in \cite{AG}).
\begin{lemma}\label{l5.1}
Let $X_n$ be a sequence of connected subsets of a metric space $X$. If
\begin{itemize}
\item[(i)] $\liminf\left(X_n\right)\neq \emptyset$;
\item[(ii)] $\bigcup X_n$ is precompact;
\end{itemize}
then $\limsup\left(X_n\right)$ is nonempty, compact and connected.
\end{lemma}
\noindent Above, $\liminf \left(X_n\right)$ and $\limsup \left(X_n\right)$ denote the set of all $x\in X$ such that any neighborhood
of $x$ intersects all but finitely many of $X_n$, infinitely many of $X_n$ respectively.\\[.3cm]
Now let us fix $k> 1$ and $\a_k=2(k-1)$. Take $n$ sufficiently large  and let $\a_k^n$ as defined in \eqref{c2-epsilon}, so that $(\a_k^n,u_{n,\a_k^n}) $
 is a  bifurcation point for problem (\ref{p-epsilon}). For simplicity we set  $\a_n:=\a_k^{n}$ and $u_n:=u_{n,\a_k^n}$. Let $\mathcal{C}(\a_n)$ be the maximal connected component which bifurcates from $(\a_n,u_n)$ in the space $Z$.
Let $\d>0$ such that in the interval $[\a_k-2\d, \a_k+2\d]$ there is not another exponent $\a_h$ with $h\neq k$. We let $\mathcal{Z}_n:=\mathcal{C}(\a_n)\cap B_{\d,X}(\a_n,u_n)$ where
$$B_{\d,Z}(\a_n,u_n):=\{(\a,v)\in Z\hbox{ such that } |\a-\a_n|+\nor v-u_n\nor_{X}<\d\}$$
where the space $X$ and its norm are as defined in \eqref{X} and \eqref{norm-x}.
Finally we denote by $\mathcal{X}_n$ the maximal connected component of $\mathcal {Z}_n$  that contains $(\a_n,u_n)$.
\begin{remark}\label{r52}
The sets $\mathcal{X}_n$ are nonempty since they contain at least $(\a_n,u_n)$, moreover $(\a_k, U_{\a_k})\in \liminf \left( \mathcal{X}_n\right)$.
\end{remark}
\noindent We have
\begin{lemma}\label{l5.2}
The set $\bigcup \mathcal{X}_n$ is precompact in $Z$.
\end{lemma}
\begin{proof}
Let $(\a_m,v_m)$ be a sequence in $\bigcup \mathcal{X}_n$. Then $(\a_m,v_m)\in \mathcal{X}_{n(m)}$ for some $n(m)>0$. We consider first the case where $n(m)\to +\infty$ as $m\to \infty$.\\
By the definition of $\mathcal{X}_{n(m)}$, the functions $v_m$ satisfy (\ref{p-epsilon}) with $\e=\e_{n(m)}$ and $\a=\a_m$. Then $\a_m\in (\a_k-2\d,\a_k+2\d)$ since $|\a_m-\a_k|\leq |\a_m-\a_{n(m)}|+|\a_{n(m)}-\a_k|\leq 2\d$ if $n(m)$ is large enough.
Hence, up to a subsequence $\a_{(m)}\to \bar \a$ and $\bar \a\in [\a_k-2\d,\a_k+2\d]$. Moreover $v_m\in \cH$ and $\nor v_m-u_{n(m)}\nor_{X}<\d$ so that $\nor v_m\nor_X<\d+\sup_{m}\nor u_{n(m)}\nor_{X}$. From \eqref{convergence}
we have that $\nor u_{n(m)}\nor_{X}\leq C$.
Then
$\nor v_m\nor _{1,2}\leq A$ and $\nor v_m\nor_{\b}\leq A$ for any $m$ for some constant $A>0$. Up to a subsequence $v_m\to \bar v$ weakly in $D^{1,2}(\R^N)$ and almost everywhere in $\R^N$. Moreover from  (\ref{p-epsilon})  we get that $v_m\to\bar v$ in $C^1_{loc}(\R^N)$, where $\bar v$ is a solution of (\ref{1}) with $\a=\bar \a$.\\
Further from Proposition
\ref{prima-stima}
there exists $C>0$, independent of $m$, such that $v_m(x)\leq \frac C{(1+|x|)^{N-2}}$.
Thus
$$\int_{\R^N}|\na v_m|^2 \, dx =C(\a_m)\int_{B_m}|x|^{\a_m}\left(v_m+\c(m)\right)^{p_m}v_m\, dx$$
where
$\c(m):=\c_{\a_m}(\e_m)$ and $p_m:=p_{\a_m}$ and passing to the limit, using $ v_m+\c(m)\leq \frac C{(1+|x|)^{N-2}}$, (see \eqref{6.2}),
$$\int_{\R^N}|\na v_m|^2 \, dx \to C(\bar \a)\int_{\R^N}|x|^{\bar \a} \bar {v}^{p_{\bar \a}+1}\, dx =\int_{\R^N}|\na \bar v|^2 \, dx.$$
Hence
$$\int_{\R^N}|\na \left(v_m-\bar v\right)|^2 \, dx=\int_{\R^N}|\na v_m|^2 \, dx-2\int_{\R^N}\na v_m\cdot \na \bar v \, dx+\int_{\R^N}|\na \bar v|^2 \, dx\to 0$$
as $m\to \infty$ so that $v_m\to v$ strongly in $D^{1,2}(\R^N)$.\\
Moreover from Proposition \ref{prima-stima} we have that
$$|v_m(x)-\bar v(x)|\leq |v_m(x)|+|\bar v(x)|\leq \frac C{(1+|x|)^{N-2}}.$$
This means that for every $\e>0$ there exists $\rho>0$ such that, for any $m$, $(1+|x|)^{\b}|v_m(x)-\bar v(x)|<\e$ if $|x|>\rho$. Then from the uniform convergence of $v_m$ to $\bar v$ on compact sets of $\R^N$ we have that $(1+|x|)^{\b}|v_m(x)-\bar v(x)|<\e$ in $B_{\rho}(0)$ if $m$ is sufficiently large, i.e $v_m\to \bar v$  in $L^{\infty}_{\beta}(\R^N)$. Then $v_m\to \bar v$ strongly in $X$.\\
If the sequence $n(m)\nrightarrow+\infty$, up to a subsequence $n(m)$ converges to $n_0\in \N$ and repeating the proof for the first case we have that a subsequence $v_{m_k}$ converges in $X$ to a solution of (\ref{p-epsilon}) for $\a=\bar \a$ and $\e=\e_{n_0}$.
\end{proof}
\begin{lemma}\label{l53}
The set $\limsup \left(\mathcal{X}_n\right)\setminus\{(\a_k, U_{\a_k} )\}$ is nonempty.
\end{lemma}
\begin{proof}
By the results of Section \ref{s5} (Theorem \ref{t-alternative} ) and regularity theorems
a global continuum bifurcates from the point $(\a_n,u_n)$ and this continuum $\mathcal{C}(\a_n)$
is either unbounded in $Z$ or it meets $\{0\}\times X$ or it achieves another bifurcation point $(\a_h^n, u_{n,\a_h^n})$ with $\a_h^n$ which is not contained in $[ \a_k-\d,\a_k+\d]$.
This implies that, on the closure of any component $\mathcal{X}_n$, there exists a point $(\bar \a_n,\bar v_n)\in \de B_{\d,Z}(\a_n,u_n)$ i.e. such that
\begin{equation}\label{5.4}
|\bar \a_n-\a_n|+\nor \bar v_n-u_n\nor_X =\d
\end{equation}
and $\bar v_n$ is a solution of (\ref{p-epsilon}) in $B_n$ for $\a=\bar \a_n$.\\
Using the bounds on $\bar v_n$ and $\bar \a_n$ and the standard regularity theorems we can pass to the limit and get that $(\bar \a_n,\bar v_n)\to (\bar \a,\bar v)$ with $\bar v$ solution of (\ref{1}) for $\a=\bar \a$, $\bar \a\in [\a_k-2\d,\a_k+2\d]$ and
$$|\bar \a-\a_k|+\nor \bar v-U_{\a_k}\nor_X =\d.$$
Then, obviously, $(\bar \a,\bar v)\in \limsup \left( \mathcal{X}_n\right)$ but $(\bar \a,\bar v)\neq (\a_k,U_{\a_k})$.
\end{proof}
\begin{theorem}\label{t5.5}
For any $k\geq 2$, the points $(\a_k,U_{\a_k})$ are nonradial bifurcation points for the curve $\mathcal{S}$ of radial solution of \eqref{1}, i.e.
\begin{equation}
  \mathcal{S}:=\left\{\begin{split}(\a,U_\a)\,\in\,  (0,+\infty)\,\times\, X\,\, \hbox{ such that } \,\,U_\a\,
    \hbox{ is the }\\
\hbox{ unique radial solution of }(\ref{1})\hbox{ such that }\,\, U_{\a}(0)=1\,
\end{split}\right\}
\end{equation}
and $X$ as defined in \eqref{X}.
\end{theorem}
\begin{proof}
Let $\a_k=2(k-1)$. 
Then, fixing $k$ as before, we consider the bifurcation points $(\a_n,u_n)$ for problem (\ref{p-epsilon}) in $B_n$ and the connected components $\mathcal{X}_n$
of the bifurcation continua in $B_{\d,Z}(\a_n,u_n)$. \\
By Remark \ref{r52} and Proposition \ref{l5.2} the sequence of sets $\mathcal{X}_n$ satisfies the hypotheses of Lemma \ref{l5.1} in the space $Z$ and hence the set
$$\mathcal{C}_k=\limsup \left( \mathcal{X}_n\right)$$
is nonempty, compact and connected. Moreover it contains $(\a_k,U_{\a_k})$ and it does not reduce only to this point by Lemma \ref{l53}.\\
If $(\bar \a,\bar v)\in \mathcal{C}_k\setminus (\a_k,U_{\a_k})$ by definition there exists a sequence of points $(\bar \a_n,v_n)\in\mathcal{X}_n$ such that $(\bar \a_n,v_n)\to (\bar \a,\bar v)$ in $Z$ and $\bar v$ is a solution of (\ref{1}) for $\a=\bar \a$ and $\bar v>0$ because $(\bar \a_n,v_n)\in B_{\d,Z}(\a_n,u_n)$ and $\d$ is small.
We want to show that $\bar v$ is a nonradial solution of (\ref{1}) for $\a=\bar \a$. To this end we need to show that $\bar v\neq U_{\l,\bar \a}$ for any $\l>0$.\\
From Proposition \ref{lambda=1} we have that $\bar v\neq U_{\l,\bar \a}$ for any $\l\neq 1$. This implies, in turn, that $\bar v$ is a nonradial solution of \eqref{1} if $\bar \a= \a_k$ since we suppose $(\bar \a,\bar v)\in \mathcal{C}_k\setminus (\a_k,U_{\a_k})$.\\
Then, the claim follows by showing that
$$\nor v_n-U_{\bar \a}\nor_X>c>0$$
for some positive constant $c$ and for any $n$ sufficiently large, for $\bar \a\neq \a_k$. Equivalently we will show that
\begin{equation}\label{stima-finale}
\nor v_n-\bar u_n\nor_X>c>0
\end{equation}
where $\bar u_n:=u_{n,\bar \a_n}$ and $\bar \a_n$ as before, for any $n$ sufficiently large and for $\bar \a\neq \a_k$, recalling that $\bar u_n\to U_{\bar \a}$ in $X$ by Proposition \ref{p3.2}.\\
To prove \eqref{stima-finale} we argue by contradiction and assume that $v_n-\bar u_n\to 0$ in $X$.
Then  $\nor v_n-\bar u_n\nor_{\infty}\to 0$ as $n\to +\infty$, and this is not possible
(see Proposition \ref{pohozaev}) since $\bar \a\neq \a_k$.  Since we have reached a contradiction \eqref{stima-finale} holds.
\end{proof}

\begin{remark}\label{rem-5.6}
\rm We remark that the bifurcation from the points $(\a_k,U_{\a_k})$ obtained in the Theorem \ref{t5.5} is indeed global. In fact we proved the existence of a closed connected set $\mathcal{C}_k$ that bifurcates from every point  $(\a_k,U_{\a_k})$.
%
Finally all solutions on the continuum $\mathcal{C}_k$ are fast decaying solutions of \eqref{1} because they decay as $\frac 1{(1+|x|)^{N-2}}$ when $|x|$ is large, from Lemma \ref{prima-stima}. This indeed is a consequence of the fact that the solutions we find are the limit, in some sense, of the solutions of the approximating problem in a ball.
\end{remark}
\begin{proof}[Proof of Theorem \ref{i17}]
Theorem \ref{t5.5} proves the existence of a continuum $\mathcal{C}_k$ of solutions of \eqref{1}, invariant with respect to $O(N-1)$ bifurcating from $(\a_k,U_{\a_k})$ with $\a_k=2(k-1)$ for any $k\geq 2$, and then proves i).\\
Moreover, when $k$ is even, repeating the proof of Theorem \ref{t5.5} using the space
$$\mathcal{H}^h:=\{v\in X\, \hbox{ s.t. }v \hbox{ is invariant by the action of }\mathcal{G}_h\}$$
for $h=1,\dots,\left[\frac N2\right]$, with $  \mathcal{G}_h$ as in \eqref{g-h}, and using Remark \ref{rem-3.10} we find $\left[\frac N2\right]$ different continua bifurcating from $(\a_k,U_{\a_k})$.
Each continuum is invariant with respect the action of $\mathcal{G}_h$ for some $h$ and then ii) follows from \eqref{g-h}.\\
Finally the decay of the solutions we find follows from Lemma \ref{prima-stima} since the continua $\mathcal{C}_k$ are bounded by construction (see also Remark \ref{rem-5.6}).
\end{proof}
\subsection{An explicit solution}
In this Section we construct an explicit branch of solutions to \eqref{1}. The idea is the same as in Theorem \ref{linearized}.

We want to reduce our problem to another one
where there is no dependence on $|x|^{\a}$ and the dimension $M=\frac{2(N+\a)}{2+\a}$, as we did in Section  \ref{s2} (see \eqref{1.8ab}).
Suppose that we have $M$ integer and consider the known solutions when $\alpha=0$ in $\R^M$,
$$U(x)= \frac1{(1+ |x-y|^2 )^\frac{M-2}2 }=
\frac1{(1+ |x|^2-2x\cdot y+|y|^2 )^\frac{M-2}2 }$$
and setting $|x|=r$  we get
\begin{equation}\label{5.6}
U(x)=\frac1{(1+r^2-2r|y|\cos\overset\wedge{xy}+|y|^2 )^\frac{M-2}2 }
\end{equation}
Proceeding as in Section \ref{s2}, we consider the transformation $r \mapsto r^{\frac{2+\a}{2}}$ and $|y|=a$. However, in this case it is not clear how it transforms the angular term in \eqref{5.6}. Thus we seek solutions of \eqref{1} in the form
\begin{equation} \label{ansatz}
u(x)=\frac1{(1+|x|^{2+\a}-2aY(x)+a^2 )^\frac{M-2}{2}}
\end{equation}
for some function $Y(x)$. The proof of Theorem \ref{linearized} suggests that the homogeneous harmonic
polynomials of degree $k=\frac{2+\a}{2}$ can be good candidates, but \eqref{ansatz} does not satisfy \eqref{1}
for every choice of $Y \in \mathbb{Y}_k$. Substituting \eqref{ansatz} into \eqref{1} one can check that $Y \in \mathbb{Y}_k$
must satisfy
\begin{equation} \label{cond}
|\nabla Y(x)|^2 = \left(\frac{2+\a}{2}\right)^2|x|^{\a}.
\end{equation}

Let us consider $\a=2$ (the other cases are more involved), we have $k=2$ and $M=1+\frac N2$ so we consider $N$ even. Hence we want to find
$Y \in \mathbb{Y}_2$ such that $|\nabla Y(x)|^2 = 4|x|^{2}$. In this case we readily see that $Y(x) = |x'|^2 - |x''|^2$ gives a solution, where we write $x \in \R^N$ as $x=(x',x'') \in \R^{\frac N2} \times \R^{\frac N2}$. Therefore the
Proposition \ref{i18} is proved.

Note that the same result can be obtained applying the method used in \cite{PS}.

\appendix
\section{Appendix. Some elementary proofs of known results}

In this appendix we give a new proof of some known results. The first one is the following inequality (see E. Lieb \cite{L83}, B Gidas and J. Spruck  \cite{GS81})
\begin{theorem}\label{a1}
 Let $u\in D^{1,2}\left(\R^N\right)$ be a radial function. Then we have that,
\begin{equation}\label{a2}
\int_{\R^N}|\nabla u|^2\ge C(\alpha,N)\left(\int_{\R^N}|x|^\alpha|u|^\frac{2N+2\alpha}{N-2}\right)^\frac{N-2}{N+\alpha}
\end{equation}
Moreover the extremal functions which achieve $C(\alpha,N)$ are unique (up to dilations) and are given by
\begin{equation}\label{a3}
U_{\l,\a}(x)=\frac {\l^{\frac{N-2}2}}{\left(1+\l^{2+\a}D|x|^{2+\a}\right)^{\frac{N-2}{2+\a}}}
\end{equation}
with $\l>0$ and a suitable $D\in\R^+$.
\end{theorem}
\begin{proof}
Let $u\in D^{1,2}\left(\R^N\right)$ be a radial function. Then we have that
\begin{align*}
\int\limits_0^\infty\left|u'(r)\right|^2r^{N-1}&=
\left(\hbox{setting }r=s^\frac2{\alpha+2}\hbox{ and }
v(s)=u\left(s^\frac2{\alpha+2}\right)\right) \\
&=\frac{\alpha+2}2\int\limits_0^\infty\left|v'(s)\right|^2s^\frac{2N-2+\alpha}{\alpha+2}.
\end{align*}
Set
\begin{equation}\label{a4}
\frac{2N-2+\alpha}{\alpha+2}=M-1
\end{equation}
which implies
$$\int\limits_0^\infty\left|u'(r)\right|^2r^{N-1}=\frac{\alpha+2}2\int\limits_0^\infty\left|v'(s)\right|^2s^{M-1}.$$
Now we use the classical Sobolev inequality (see Lemma 2 in \cite{TA76}) and we get
\begin{align*}
&\frac{\alpha+2}2\int\limits_0^\infty\left|v'(s)\right|^2s^{M-1}\ge\frac{\alpha+2}2S(M)\left(\int\limits_0^\infty|v(s)|^\frac{2M}{M-2}s^{M-1}\right)
^{\frac{M-2}M}\\
&=\left(\frac{\alpha+2}2\right)^{\frac{2M-2}M}S(M)\!\left(\int\limits_0^\infty|u(r)|^\frac{2M}{M-2}r^\frac{M(\alpha+2)-2}2\right)^{\frac{M-2}M}.
\end{align*}
Here $S(M)=M(M-2)\left[\frac{\left(\Gamma(\frac M2)\right)^2}{2\Gamma(M)}\right]^\frac2M$ (see \cite{TA76}).\\
From \eqref{a4} we deduce that $\frac{2M}{M-2}=\frac{2N+2\alpha}{N-2}$ and  $\frac{M(\alpha+2)-2}2=N-1+\alpha$.  So we get
 $$\int\limits_0^\infty\left|u'(r)\right|^2r^{N-1}\ge\left(\frac{\alpha+2}2\right)^\frac{2N-2+\alpha}{N+\alpha}S\left(\frac{2N+2\alpha}{\alpha+2}\right)\left(\int\limits_0^\infty r^\alpha|u(r)|^\frac{2N+2\alpha}{N-2}r^{N-1}\right)^\frac{N-2}{N+\alpha}$$
which proves \eqref{a2} with
$$C(\alpha,N)= \left(\frac{\alpha+2}2\right)^\frac{2N-2+\alpha}{N+\alpha} S\left(\frac{2N+2\alpha}{\alpha+2}\right) \left( \frac{2 \pi^{\frac N 2}}{\Gamma\left( \frac N 2 \right)} \right)^{\frac{2 + \alpha}{N + \alpha}}  $$\\
Moreover, from the previous inequalities, we also get that the extremal functions are obtained as
$$\int\limits_0^\infty\left|v'(s)\right|^2s^{M-1}=S(M)\left(\int\limits_0^\infty|v(s)|^\frac{2M}{M-2}s^{M-1}\right)
^{\frac{M-2}M}$$
It is well-known that $v_\mu(s)=\frac{\mu^\frac{M-2}2}{\left(1+D^2\mu^2s^2\right)^\frac{M-2}2}$ for some positive constant $D$ and for any $\mu\in\R^+$. By \eqref{a4} we get that the functions $U_{\l,\a}(x)=v\left(s^\frac{\alpha+2}2\right)=
\frac {\mu^\frac{N-2}{2+\a}}{\left(1+D^2\mu^2|x|^{2+\a}\right)^\frac{N-2}{2+\a}}$ are extremal for the inequality  \eqref{a2}. Setting $\mu^2=\l^{2+\alpha}$ we have \eqref{a3}.
\end{proof}
Next result is a short proof of the existence and uniqueness result of the radial solution founded by W. M. Ni in \cite{N82}.
\begin{theorem}\label{a5}
Let $B_1$ the unit ball in $\mathbb{R}^N$ with $N\ge3$. Then the problem
\begin{equation}\label{a6}
\left\{\begin{array}{ll}
-\Delta u=|x|^{\a}u^p  & \hbox{ in }B_1\\
u>0& \hbox{ in }B_1\\
u=0& \hbox{ on }\partial B_1
\end{array}\right.
\end{equation}
admits a unique radial solution for $1<p<\frac{N+2+2\a}{N-2}$ and any $\a\ge0$.
\end{theorem}
\begin{proof}
Since $u$ is radial we have that it satisfies,
\begin{equation}\label{a7}
\left\{\begin{array}{ll}
-u''-\frac{N-1}ru'=r^{\a}u^p  & \hbox{ in }(0,1)\\
u>0&  \hbox{ in }(0,1)\\
u'(0)=u(1)=0&
\end{array}\right.
\end{equation}
As in the previous theorems, let us consider the transformation
$$r=s^\frac2{\alpha+2}\quad\hbox{ and }\quad
v(s)=u\left(s^\frac2{\alpha+2}\right).$$
Then \eqref{a7} becomes, setting $M=\frac{2(N+\a)}{2+\a}$
\begin{equation}\label{a8}
\left\{\begin{array}{ll}
 -v''(s) - \frac{M-1}{r}v'(s)=\frac4{(2+\a)^2}v^p\,& \text{ in } (0,1) \\
v>0&  \hbox{ in }(0,1)\\
v'(0)=v(1)=0
\end{array}\right.
\end{equation}
Note that we have $v'(0)=0$ since $v'(s)=\frac{2}{\alpha+2} s^{-\frac\a{\a+2}}u'\left(s^\frac2{\alpha+2}\right)$ and by \eqref{a7} $u'(r)= O(r^{\a+1})$ near $r=0$. Moreover, since $1<p<\frac{N+2+2\a}{N-2}$ we get $1<p<\frac{M+2}{M-2}$.\\
Since the space $H=\left\{u\in H^1(0,1)\hbox{ such that }u(1)=0\right\}$ equipped with the norm $||u||_H^2=\int_0^1|u'(s)|^2s^{M-1}ds$ is compactly embedded in
$L=\{u\in L^{p+1}(0,1)$ such that $\int_0^1|u(s)|^{p+1}s^{M-1}ds\},$
for any $M>0$ and  $1<p<\frac{M+2}{M-2}$, we have the existence of the solution.\\
The ODE \eqref{a8} was considered in \cite{GNN79}, (Sec. 2.8, p. 224) when $1<p<\frac{M+2}{M-2}$ and it was proved that it has a {\em unique} solution.  In \cite{GNN79} only the case $M\in\N$ was considered, but it is easy to see that the same proof applies in the general case $M\in\R$.
So the same uniqueness result holds for problem \eqref{a7} and the claim follows.
\end{proof}

\end{document}